\documentclass[11pt]{amsart}

\usepackage{amscd}
\usepackage{amsmath, amssymb}
\usepackage{amsfonts}
\newcommand{\de}{\partial}
\newcommand{\db}{\overline{\partial}}

\newcommand{\ddbar}{\sqrt{-1} \partial \overline{\partial}}
\newcommand{\Ric}{\mathrm{Ric}}
\newcommand{\ov}[1]{\overline{#1}}
\newcommand{\mn}{\sqrt{-1}}

\newcommand{\tr}[2]{\textrm{tr}_{#1}{#2}}
\newcommand{\ti}[1]{\tilde{#1}}
\newcommand{\vp}{\varphi}

\newcommand{\ve}{\varepsilon}

\renewcommand{\leq}{\leqslant}
\renewcommand{\geq}{\geqslant}
\renewcommand{\le}{\leqslant}
\renewcommand{\ge}{\geqslant}

\begin{document}
\newtheorem{claim}{Claim}
\newtheorem{theorem}{Theorem}[section]
\newtheorem{conjecture}[theorem]{Conjecture}
\newtheorem{lemma}[theorem]{Lemma}
\newtheorem{corollary}[theorem]{Corollary}
\newtheorem{proposition}[theorem]{Proposition}
\newtheorem{question}[theorem]{question}
\newtheorem{defn}[theorem]{Definition}
\numberwithin{equation}{section}

\newenvironment{example}[1][Example]{\addtocounter{remark}{1} \begin{trivlist}
\item[\hskip
\labelsep {\bfseries #1  \thesection.\theremark}]}{\end{trivlist}}
\title[Hermitian metrics and Monge-Amp\`ere equations]
{Hermitian metrics, $(n-1, n-1)$ forms and Monge-Amp\`ere equations$^*$}

\author[V. Tosatti]{Valentino Tosatti}
\thanks{$^{*}$Research supported in part by NSF grants DMS-1236969, DMS-1308988 and DMS-1332196.  The first named-author is supported in part by a Sloan Research Fellowship.}
\address{Department of Mathematics, Northwestern University, 2033 Sheridan Road, Evanston, IL 60208}
\author[B. Weinkove]{Ben Weinkove}

\begin{abstract}  We show existence of unique smooth solutions to the Monge-Amp\`ere equation for $(n-1)$-plurisubharmonic functions on Hermitian manifolds, generalizing previous work of the authors.  As a consequence we obtain Calabi-Yau theorems  for Gauduchon and strongly Gauduchon metrics on a class of non-K\"ahler manifolds:  those satisfying the Jost-Yau condition known as Astheno-K\"ahler.

Gauduchon conjectured in 1984 that a Calabi-Yau theorem for Gauduchon metrics holds on \emph{all} compact complex manifolds.  We discuss another
Monge-Amp\`ere equation, recently introduced by  Popovici, and show that the full Gauduchon conjecture can be reduced to a second order estimate of
  Hou-Ma-Wu type.
\end{abstract}
\maketitle

\section{Introduction}

The complex Monge-Amp\`ere equation
\begin{equation} \label{yau}
(\omega + \ddbar u)^n = e^F \omega^n, \quad \omega+\ddbar u >0,
\end{equation}
on a compact K\"ahler manifold $(M, \omega)$ was solved by Yau in the 1970's \cite{Ya} and has played an ubiquitous  role in K\"ahler geometry ever since.   Here $F$ is a given smooth function on $M$, normalized so that $\int_M e^F \omega^n = \int_M \omega^n$.  Yau's Theorem, conjectured in the 1950's by Calabi, is that (\ref{yau}) has a unique solution $u$ with $\sup_M u=0$.

For a general Hermitian metric $\omega$, the complex Monge-Amp\`ere equation was solved  in full generality by the authors in \cite{TW2} (see also \cite{Ch, GL, TW1}).  In this case, there exists a unique pair $(u,b)$ with $u$ a smooth function satisfying $\sup_M u=0$ and $b$ a constant such that
\begin{equation} \label{cxma}
(\omega + \ddbar u)^n = e^{F+b} \omega^n, \quad \omega+\ddbar u >0.
\end{equation}
This equation has applications to the study of cohomology classes and notions of positivity
 on complex manifolds \cite{TW2a, Chi, Xi}.

Let $M$ be a compact complex manifold of dimension $n>2$.  We
 can consider metrics $\omega$ satisfying conditions which are weaker than K\"ahler.  There is a bijection from the space of positive definite $(1,1)$ forms to positive definite $(n-1,n-1)$ forms, given by
\begin{equation} \label{bijection}
\omega \mapsto \omega^{n-1}.
\end{equation}
A key point is that closedness conditions on $\omega^{n-1}$ impose fewer equations than the same conditions on $\omega$.
Indeed, Gauduchon \cite{Ga0} showed that every Hermitian metric is conformal to a metric $\omega$ satisfying
$$\partial \ov{\partial} (\omega^{n-1}) =0,$$
a condition now known as \emph{Gauduchon}.  A stronger condition, introduced recently by Popovici \cite{Po},
$$\ov{\partial} (\omega^{n-1}) \  \textrm{is $\partial$-exact},$$
 is known as \emph{strongly Gauduchon}.

There are natural Monge-Amp\`ere equations associated to these conditions, obtained by replacing $(1,1)$ forms by $(n-1,n-1)$ forms.
Our first result is about an equation which we call the \emph{Monge-Amp\`ere equation for $(n-1)$-plurisubharmonic functions}.

\begin{theorem} \label{mainthm} Let $M$ be a compact complex manifold with two Hermitian metrics $\omega_0$ and $\omega$.  Let $F$ be a smooth function.  Then there exists a unique pair $(u,b)$ with $u$ a smooth function on $M$ and $b$ a constant such that
\begin{equation} \label{ma1}
\det \left( \omega_0^{n-1} + \ddbar u \wedge \omega^{n-2} \right) = e^{F+b} \det \left( \omega^{n-1} \right),
\end{equation}
with $$\omega_0^{n-1} + \ddbar u \wedge \omega^{n-2} > 0, \quad \sup_M u=0.$$
\end{theorem}

This generalizes a recent result of the authors \cite{TW4}, where it was proved for $\omega$ K\"ahler, a conjecture of Fu-Xiao \cite{FX} (see also \cite{FWW1, Po2}).  The equation (\ref{ma1}) for $\omega$ K\"ahler was first introduced by Fu-Wang-Wu \cite{FWW1, FWW2} who proved uniqueness and a number of other properties, including  existence in the case when $\omega$ satisfies an assumption on curvature.  The authors learned about equations of the type (\ref{ma1}) from J.-P. Demailly in relation to questions about strongly Gauduchon metrics.

 Note that a function $u$ satisfying
$$\omega_0^{n-1} + \ddbar u \wedge \omega^{n-2} > 0$$
is called \emph{$(n-1)$-plurisubharmonic} with respect to $\omega_0^{n-1}$ and $\omega$.  The reason for this terminology is as follows.  If $\omega$ is the Euclidean metric on $\mathbb{C}^n$ then the condition
$$\ddbar u \wedge \omega^{n-2} \ge 0,$$
is the statement that $u$ is $(n-1)$-plurisubharmonic in the sense of Harvey-Lawson \cite{HL}.  It is equivalent to the assertion that $u$ is subharmonic when restricted to every complex $(n-1)$-plane in $\mathbb{C}^n$.  Thus equation (\ref{ma1}) can be regarded as the Monge-Amp\`ere equation for $(n-1)$-plurisubharmonic functions on Hermitian manifolds.

We now describe briefly a technical innovation which we use to prove Theorem \ref{mainthm}.    A key difficulty in the case of Hermitian metrics arises in the second order estimate of $u$ (Theorem \ref{theoremC2} below).  There  are new
terms of the form $T * D^3u$, where $T$ is the torsion of $\omega$ and $D^3u$ an expression involving three derivatives of $u$.  To control these terms, the idea is to add a small multiple of the quantity $g^{p\ov{q}} \eta_{i\ov{q}} \eta_{p\ov{j}}$, the ``square of the tensor $\eta$ with respect to the metric $g$'', where $\eta_{i\ov{j}}$ is the same tensor as defined in \cite{TW4}.  By doing so we obtain positive terms which can bound most of the $T * D^3u$ terms.   A fortunate cancellation of two remaining torsion terms allows the argument to work (cf. (\ref{mystery1}) and (\ref{Line14}) below).

We now explain how equation (\ref{ma1}) is related to Gauduchon and strongly Gauduchon metrics.  In fact,
Theorem \ref{mainthm} gives Calabi-Yau theorems for these metrics for a certain class of non-K\"ahler manifolds.   Suppose that $\omega$ satisfies the condition
\begin{equation} \label{ak}
\partial \ov{\partial} (\omega^{n-2})=0.
\end{equation}
Then
 if $\omega_0$ is Gauduchon  and $u$ a smooth function, a metric $\omega_u$ whose $(n-1)$th power is given by
$$\omega_u^{n-1} = \omega_0^{n-1} + \ddbar u \wedge \omega^{n-2}>0$$
is also Gauduchon.  The same assertion holds if we replace ``Gauduchon'' by ``strongly Gauduchon''.

A metric $\omega$ satisfying (\ref{ak}) is called \emph{Astheno-K\"ahler}. Astheno-K\"ahler metrics were introduced by Jost-Yau in \cite{JY}. See also \cite{FT, LYZ, Ma, MT} for more on these metrics, and for examples of Astheno-K\"ahler manifolds without K\"ahler metrics.

We obtain:

\begin{corollary}\label{cor1}
Let $M$ be a compact complex manifold equipped with an Astheno-K\"ahler metric $\omega$.  Let $\omega_0$ be a Gauduchon (resp. strongly Gauduchon) metric on $M$ and $F'$ a smooth function on $M$.  Then there exists a unique constant $b'$ and a unique Gauduchon (resp. strongly Gauduchon) metric, which we write as $\omega_u$, with
$$\omega_u^{n-1} = \omega_0^{n-1} + \ddbar u \wedge \omega^{n-2}$$
for some smooth function $u$, solving the Calabi-Yau equation
\begin{equation} \label{mma1}
\omega_u^n = e^{F'+b'} \omega^n.
\end{equation}
\end{corollary}

This result was proved by the authors in \cite{TW4} under the stronger assumption that $\omega$ is K\"ahler.  In fact, when $\omega$ is K\"ahler
then the result of \cite{TW4} also gives a Calabi-Yau Theorem for \emph{balanced} metrics  \cite{Mi} (see also \cite{FX, FY, LY, St, To}).

Gauduchon made the following conjecture in 1984 \cite[IV.5]{Ga}:

\begin{conjecture} \label{conjG}
Let $M$ be a compact complex manifold.  Let $\psi$ be a closed real $(1,1)$ form on $M$ with $[\psi]=c_1^{\mathrm{BC}}(M)$ in $H^{1,1}_{\mathrm{BC}}(M,\mathbb{R})$. Then there exists a Gauduchon metric $\ti{\omega}$ on $M$ with
\begin{equation} \label{gCY}
\Ric(\ti{\omega})=\psi.
\end{equation}
\end{conjecture}
 Here $\textrm{Ric}(\tilde{\omega})$ is the Chern-Ricci curvature, given locally by
 $$\Ric(\ti{\omega}) = - \ddbar \log \tilde{\omega}^n,$$
 and $H^{1,1}_{\mathrm{BC}}(M,\mathbb{R})$ is the Bott-Chern cohomology group of $d$-closed real $(1,1)$ forms modulo $\de\db$-exact ones.
 Gauduchon's conjecture is a natural extension of the celebrated Calabi conjecture \cite{Ya} to compact complex manifolds.  In complex dimension 2, Gauduchon's conjecture follows from the result of Cherrier \cite{Ch} (see  \cite{TW1} for an alternative proof) since in that case it follows from the solution of the complex Monge-Amp\`ere equation (\ref{cxma}).

As a consequence of Corollary \ref{cor1}, we can prove Gauduchon's conjecture if $M$ admits an Astheno-K\"ahler metric.

\begin{corollary} \label{cor2}
Let $M$ be a compact complex manifold equipped with an Astheno-K\"ahler metric $\omega$. Then Conjecture \ref{conjG} holds on $M$.
\end{corollary}

In the special case when $(M, \omega)$ is K\"ahler then Yau's solution of the Calabi conjecture already gives
 $\tilde{\omega}$ K\"ahler (and hence Gauduchon) satisfying (\ref{gCY}) in every K\"ahler class.  The previous result of the authors in \cite{TW4} shows that in this case one can also find a Gauduchon metric $\tilde{\omega}$ satisfying (\ref{gCY}) where $\tilde{\omega}^{n-1}$ has the form $\tilde{\omega}^{n-1} = \omega_0^{n-1} + \ddbar u \wedge \omega^{n-2}$ for any given Gauduchon metric $\omega_0$.

Next, we discuss a different Monge-Amp\`ere equation which is closely related to Gauduchon's conjecture.  This equation is a modification of (\ref{ma1}) and
 first appeared\footnote{The authors independently discovered the equation (\ref{manew}) in June 2013 and completed Theorems \ref{Linftythm} and \ref{thmreduce} of this paper before \cite{Po3} was posted on the arXiv.} in the very recent preprint of Popovici \cite{Po3}.  The
solution of this equation would solve the full conjecture of Gauduchon.

Let $\omega_0$ be any Hermitian metric and $\omega$ a Gauduchon metric.  To deal with the fact that $\partial \ov{\partial} \omega^{n-2} \neq 0$ in general, we consider the following $(n-1,n-1)$ form
$$\Phi_u=\omega_0^{n-1}+\ddbar u\wedge\omega^{n-2}+\mathrm{Re}\left(\mn\de u\wedge \db(\omega^{n-2})\right).$$
The point of the definition of $\Phi_u$ is that if $\omega_0$ is Gauduchon (strongly Gauduchon) and $\Phi_u>0$ then the $(n-1)$th root of $\Phi_u$ is also Gauduchon (strongly Gauduchon).
Indeed, the
 $(n-1,n-1)$ form $\beta_u= \ddbar u\wedge\omega^{n-2}+\mathrm{Re}\left(\mn\de u\wedge \db(\omega^{n-2})\right)$ is
$\de\db$-closed.  Moreover $\ov{\partial}\beta_u$ is $\partial$-exact.   In fact, $\beta_u$ is the $(n-1,n-1)$-part of the $d$-exact $(2n-2)$ form $d(d^c u\wedge\omega^{n-2})$.

Replacing $\omega^{n-1} + \ddbar u \wedge \omega^{n-2}$ in (\ref{ma1}) by $\Phi_u$ gives a   Monge-Amp\`ere type equation.
 We conjecture that this equation can always be solved (Popovici posed the same statement as a question \cite{Po3}).

\pagebreak[3]
\begin{conjecture} \label{conjtw}
Let $M$ be a compact complex manifold with a Hermitian metric $\omega_0$ and a Gauduchon metric $\omega$.  Let $F$ be a smooth function.  Then there exists a unique pair $(u,b)$ with $u$ a smooth function on $M$ and $b$ a constant solving the equation
\begin{equation} \label{manew}
\det \left( \Phi_u \right)  =  e^{F+b} \det \left( \omega^{n-1} \right),
\end{equation}
with
$$\Phi_u=\omega_0^{n-1} + \ddbar u\wedge\omega^{n-2}+\mathrm{Re}\left(\mn\de u\wedge \db(\omega^{n-2})\right) >0$$
and $\sup_M u=0.$
\end{conjecture}

This conjecture would imply the conjecture of Gauduchon (see Section \ref{sectionproofs}).  Popovici \cite{Po3} proved uniqueness for this equation and computed its linearization.

We now describe some progress towards Conjecture \ref{conjtw}.  First, we establish the following \emph{a priori} $L^{\infty}$ estimate for (\ref{manew}).

\begin{theorem} \label{Linftythm}  In the setting of Conjecture \ref{conjtw},
let $u$ be a smooth solution of (\ref{manew}).  Then there exists a uniform constant $C$ depending only on $\omega$, $\omega_0$ and $\sup_M |F|$ such that
$$\| u \|_{L^{\infty}} \le C.$$
\end{theorem}

In the proof of this result we again make use of a cancellation between some bothersome torsion terms (see (\ref{calc2}) below). 
As a consequence of this and suitable modifications of arguments of \cite{TW4}, it follows that Conjecture \ref{conjtw} can be reduced to an \emph{a priori} second order estimate.

\begin{theorem} \label{thmreduce} Let $u$ solve (\ref{manew}) as above.  If there exists a uniform constant $C$, depending only on $\omega$, $\omega_0$ and bounds for $F$ such that
\begin{equation} \label{hmwtype}
\Delta u \le C( \sup_M | \nabla u|^2 +1),
\end{equation}
then  Conjecture \ref{conjtw} (and hence also Conjecture \ref{conjG}) holds.
\end{theorem}
An estimate of the type (\ref{hmwtype}) was proved by Hou-Ma-Wu \cite{HMW} for the complex Hessian equations, and similar arguments are used for Monge-Amp\`ere type equations in \cite{TW4} and in the proof of Theorem \ref{mainthm}.  We expect that (\ref{hmwtype}) does indeed hold for solutions of (\ref{manew}), but torsion terms arising from quantity $\textrm{Re}(\sqrt{-1}\partial u \wedge \ov{\partial} (\omega^{n-2}))$ have so far thwarted our attempts to prove it.

\medskip

The outline of the paper is as follows.  In Sections \ref{sectionzero} and \ref{sectionsecond} we prove zero and second order \emph{a priori} estimates for a solution $u$ of the equation (\ref{ma1}).  From these estimates together with arguments adapted from \cite{TW4}, we deduce in Section \ref{sectionproofs} the proofs of Theorem \ref{mainthm}, Corollary \ref{cor1} and Corollary \ref{cor2}.  In Sections \ref{sectionnew} and \ref{sectionreduce} we discuss the Monge-Amp\`ere equation (\ref{manew}) and prove Theorems \ref{Linftythm} and  \ref{thmreduce} respectively.

\medskip

The authors are very grateful to Wei Sun for pointing out a couple of errors in a previous version of this paper.

\section{Zero order estimate} \label{sectionzero}

In this section we prove an \emph{a priori} zeroth order estimate for a solution of the Monge-Amp\`ere equation (\ref{ma1}).  The argument we give is similar to the one when $\omega$ is K\"ahler \cite{TW4}, except that we have to make use of arguments of \cite{Ch, TW1, TW2} to deal with torsion terms that arise when integrating by parts.

First,  we introduce some notation which will be used throughout the paper.
Given a Hermitian metric $g$ on $M$, which we write in local holomorphic coordinates as $(g_{i\ov{j}}) >0$, we define its associated $(1,1)$ form $\omega$ to be
$$\omega= \sqrt{-1} g_{i\ov{j}} dz^i \wedge d\ov{z}^j.$$
Here we are summing in the repeated indices $i$ and $j$.  We commonly refer to both $\omega$ and $g$ as ``the Hermitian metric''.

Equation (\ref{ma1}) is an equation of $(n-1, n-1)$ forms.
Given a smooth function $u$ on $M$ we write $\Psi_u$ for the positive definite $(n-1,n-1)$ form
$$\Psi_u=\omega_0^{n-1}+\ddbar u\wedge\omega^{n-2}>0.$$
Recall that an $(n-1, n-1)$ form $\Psi$ is positive definite if
$$\Psi \wedge \sqrt{-1} \gamma \wedge \ov{\gamma} > 0,$$
for all nonzero $(1,0)$ forms $\gamma$.  For $\omega$ as above, we define the determinant of $\omega^{n-1}$ to be
$$\det (\omega^{n-1}) = (\det g)^{n-1},$$
and  this defines the determinant of a general positive definite $(n-1,n-1)$ form via (\ref{bijection}).

As in \cite{TW4}, to solve (\ref{ma1}) we will first apply the Hodge star operator $*$ of $\omega$ to rewrite it as an equation of $(1,1)$ forms.  Define
$$\tilde{\omega}: = \frac{1}{(n-1)!}*\Psi_u, \quad \omega_h=\frac{1}{(n-1)!}*\omega_0^{n-1}.$$
Then $\tilde{\omega}$ and $\omega_h$ are Hermitian metrics on $M$.   Write
$$\tilde{\omega} = \sqrt{-1} \tilde{g}_{i\ov{j}} dz^i \wedge d\ov{z}^{j}, \quad \omega_h = \sqrt{-1} h_{i\ov{j}} dz^i \wedge d\ov{z}^{j}.$$
 We have (see Section 2 of \cite{TW4})
\begin{equation}\label{metr}
\ti{\omega}=\omega_h+\frac{1}{n-1}\left((\Delta u)\omega-\ddbar u\right)>0,
\end{equation}
and the equation (\ref{ma1}) becomes
\begin{equation} \label{ma2}
\tilde{\omega}^n = e^{F+b} \omega^n,
\end{equation}
where we recall that we normalize $u$ by $\sup_M u=0$.
Taking the trace of \eqref{metr} we see that
\begin{equation}\label{lapl}
\tr{\omega}{\ti{\omega}}=\tr{\omega}{\omega_h}+\Delta u,
\end{equation}
and therefore
\begin{equation}\label{ddb}
\ddbar u =(n-1)\omega_h+ (\tr{\omega}{\ti{\omega}}-\tr{\omega}{\omega_h})\omega-(n-1)\ti{\omega}.
\end{equation}

Before we state the $L^{\infty}$ estimate for $u$, we first note that, by looking at the points where $u$ achieves its maximum and its minimum, we immediately conclude that
\begin{equation}\label{boundb}
|b|\leq\sup_M|F|+C,
\end{equation}
for a uniform constant $C$ (which depends only on $(M,\omega_0)$ and $\omega$).
 Let $u$ solve  (\ref{ma2}) with $\sup_M u=0$, where we recall that $\tilde{\omega}$ is given by (\ref{metr}).  We will prove:

\begin{theorem}\label{linfty}
There exists a constant $C$ which depends only on $\omega_0,\omega$ and $\sup_M|F|$, such that
$$\| u \|_{L^{\infty}} \le C.$$
\end{theorem}
\begin{proof}
We make use of the following lemma from
 \cite{TW4} (see Lemma 3.2 there):
 \begin{lemma} \label{lemmatw} Assume that Hermitian metrics $\omega, \omega_h, \omega_0$ and $\tilde{\omega}$ are related by
 $$\omega_h = \frac{1}{(n-1)!} * \omega_0^{n-1} \quad  \textrm{and} \quad  \tilde{\omega}^n = e^{\ti{F}} \omega^n,$$
for $*$ the Hodge star operator of $\omega$. Define
 \begin{equation}\label{1alpha}
 \alpha = (n-1) \omega_h + (\emph{tr}_{\omega}{\tilde{\omega}} - \emph{tr}_{\omega}{\omega_h})\omega -(n-1)\tilde{\omega}.
 \end{equation}
 Then
 $$\alpha \wedge (2\omega_0^{n-1} + \alpha \wedge \omega^{n-2}) \le C \omega^n,$$
 for a constant $C$ depending only on $\sup_M |\ti{F}|$, $\omega$ and $\omega_0$ (in particular, $C$ is independent of $\tilde{\omega}$).
 \end{lemma}

Note that the lemma is a pointwise statement about Hermitian metrics.  Here, we have
$\ddbar u = \alpha$ as defined by (\ref{1alpha}) and $\ti{F}=F+b$. Using \eqref{boundb} and this lemma, we conclude that
\begin{equation}\label{need}
\ddbar u \wedge (2\omega_0^{n-1}+\ddbar u\wedge\omega^{n-2})\leq C\omega^n,
\end{equation}
for a constant $C$ as in the statement of Theorem \ref{linfty} (we will refer to such constants as uniform constants).

We  make use of  \eqref{need} in a Moser iteration argument. We compute for $p$ sufficiently large,
\begin{equation} \label{c1m}
\begin{split}
\lefteqn{  \int_M e^{-pu} \ddbar u \wedge \left( 2\omega_0^{n-1}+\ddbar u\wedge\omega^{n-2}  \right) } \\
= {} & p\int_M e^{-pu} \sqrt{-1} \partial u \wedge \ov{\partial} u \wedge  \left( 2\omega_0^{n-1}+\ddbar u\wedge\omega^{n-2}  \right)  \\
&+\int_M  e^{-pu}\mn \, \db u \wedge \de \left( 2\omega_0^{n-1}+\ddbar u\wedge\omega^{n-2}   \right) \\
=: {} & (A) + (B).
\end{split}
\end{equation}
We use the inequality
$$\omega_0^{n-1}+\ddbar u\wedge\omega^{n-2}>0,$$
to see that
\begin{equation} \label{c2m}
(A) \ge  \frac{1}{Cp} \int_M \sqrt{-1} \partial e^{-\frac{pu}{2}} \wedge \ov{\partial} e^{-\frac{pu}{2}} \wedge  \omega^{n-1},
\end{equation}
for a uniform $C>0$.
For $(B)$, using $e^{-pu} \sqrt{-1} \, \ov{\partial}u = -\frac{1}{p} \sqrt{-1} \ov{\partial} e^{-pu}$, we have, for $p$ sufficiently large,
\begin{equation} \label{c3m}
\begin{split}
(B)
 = {} & \frac{1}{p}\int_M e^{-pu}\mn\db\de  \left( 2\omega_0^{n-1}+\ddbar u\wedge\omega^{n-2}  \right) \\
= {}  & \frac{1}{p}\int_M e^{-pu} \ddbar u \wedge \ddbar \omega^{n-2}  + \frac{2}{p} \int_M e^{-pu} \ddbar \omega_0^{n-1} \\
={} & \int_M e^{-pu}\sqrt{-1} \partial u \wedge \ov{\partial} u \wedge \ddbar \omega^{n-2} + \frac{2}{p} \int_M e^{-pu} \ddbar \omega_0^{n-1} \\
= {} & \frac{4}{p^2} \int_M \sqrt{-1} \partial e^{-\frac{pu}{2}} \wedge \ov{\partial} e^{-\frac{pu}{2}} \wedge \ddbar \omega^{n-2}  + \frac{2}{p} \int_M e^{-pu} \ddbar \omega_0^{n-1} \\
\ge {} & - \frac{1}{2Cp} \int_M \sqrt{-1} \partial e^{-\frac{pu}{2}} \wedge \ov{\partial} e^{-\frac{pu}{2}} \wedge \omega^{n-1} - \frac{C'}{p} \int_M e^{-pu} \omega^n,
\end{split}
\end{equation}
where the constant $C$ is the same as the one in (\ref{c2m}).
On the other hand, multiplying \eqref{need} by $e^{-pu}$ and integrating, we have
\begin{equation} \label{c4m}
\begin{split}
\int_M e^{-pu} \ddbar u \wedge \left( 2\omega_0^{n-1}+\ddbar u\wedge\omega^{n-2} \right)
 \le {} & C \int_M e^{-pu} \omega^n.
\end{split}
\end{equation}
Combining (\ref{c1m}), (\ref{c2m}), (\ref{c3m}), (\ref{c4m}) we conclude that for $p$ sufficiently large we have the ``Cherrier-type'' inequality (see \cite{Ch, TW2})
\begin{equation}\label{cherr}
\int_M |\partial e^{-\frac{pu}{2}}|_g^2\omega^n\leq Cp\int_M e^{-pu}\omega^n.
\end{equation}
In \cite[Lemma 2.2]{TW2} we proved that \eqref{cherr} implies that
$|\{u\leq \inf_M u+C_0\}|\geq \delta>0$, for some uniform $\delta, C_0$.
We give two proofs of how to derive a $C^0$ bound on $u$ from this.

For the first proof, note that by Gauduchon's Theorem \cite{Ga0}, there exists a smooth function $\sigma$ on $M$ so that $g'=e^{\sigma} g$ is Gauduchon.  Write $\Delta'=g'^{i\ov{j}} \partial_i \partial_{\ov{j}}$ for the Laplacian of $g'$ acting on functions.
We use the existence of a Green's function for $\Delta'$: since $\Delta'$ is an elliptic second order differential operator, with the kernel consisting of just constants, standard linear PDE theory (see e.g. \cite[Appendix A]{AS}) shows that there exists a green function $G$ for $\Delta'$ which satisfies
$G(x,y)\geq -C, \|G(x,\cdot)\|_{L^1}\leq C,$ and
$$u(x)=\frac{1}{\int_M \omega'^n}\int_M u\omega'^n -\int_M \Delta' u(y) G(x,y)\omega'^n(y),$$
for all $u$ and $x$. Since $\omega'$ is Gauduchon we have $\int_M \Delta' u\omega'^n=0$.
Therefore, we are free to add a large constant to $G(x,y)$ to make it nonnegative, while preserving the same Green formula. Using
$$\Delta' u = e^{-\sigma} \Delta u \ge - e^{-\sigma} \tr{\omega}{\omega_h}\geq -C,$$ and $\sup_M u=0$, we immediately deduce that $\int_M (-u)\omega'^n\leq C.$
But then we have
$$-\delta\inf_M u\leq\int_{\{u\leq \inf_M u+C_0\}}(-u+C_0)\leq C.$$

The second proof uses Moser iteration, in the spirit of \cite{TW2}.
Let $v=u-\inf_M u$, so that $\Delta' v \geq -C$.
For any $p\geq 1$ compute
\begin{equation}\label{comp}
\begin{split}
\int_M |\de v^{\frac{p+1}{2}}|^2_{g'}\omega'^n&=\frac{n(p+1)^2}{4}\int_M \mn v^{p-1}\de v\wedge\db v\wedge\omega'^{n-1}\\
&=\frac{n(p+1)^2}{4p}\int_M \mn\de v^p\wedge\db v\wedge\omega'^{n-1}\\
&=\frac{(p+1)^2}{4p}\int_M  v^p(-\Delta' v)\omega'^n+\frac{n(p+1)}{4p}\int_M \mn\db v^{p+1}\wedge\de(\omega'^{n-1})\\
&=\frac{(p+1)^2}{4p}\int_M  v^p(-\Delta' v)\omega'^n\\
&\leq C\frac{(p+1)^2}{4p}\int_M v^p\omega'^n,
\end{split}
\end{equation}
using the Gauduchon condition $\de\db(\omega'^{n-1})=0$.  Since $\sigma$ is bounded we can switch from $\omega$ to $\omega'$ and obtain:
\begin{equation}\label{comp2}
\begin{split}
\int_M |\de v^{\frac{p+1}{2}}|^2_{g}\omega^n
&\leq C\frac{(p+1)^2}{4p}\int_M v^p\omega^n,
\end{split}
\end{equation}

A standard Moser iteration argument implies that
$$-\inf_M u=\sup_M v\leq C\int_M v\omega^n+C.$$
To bound $\|v\|_{L^1}$, use the Poincar\'e inequality and \eqref{comp2} with $p=1$ to get
$$\|v-\underline{v}\|_{L^2}\leq C\left(\int_M|\de v|^2_g\omega^n\right)^{\frac{1}{2}}\leq C'\|v\|_{L^1}^{\frac{1}{2}},$$
where $\underline{v}$ is the average of $v$ with respect to $\omega^n$. But also
$$\frac{\delta}{\int_M\omega^n}\|v\|_{L^1}=\delta \underline{v}\leq \int_{\{v\leq C_0\}}\underline{v}\omega^n\leq \int_{\{v\leq C_0\}} (|v-\underline{v}|+C_0)\omega^n\leq \|v-\underline{v}\|_{L^1}+C,$$
hence
$$\|v\|_{L^1}\leq C(\|v-\underline{v}\|_{L^1}+1)\leq C(\|v-\underline{v}\|_{L^2}+1)\leq C(\|v\|_{L^1}^{\frac{1}{2}}+1),$$
which implies that $\|v\|_{L^1}\leq C$, and we are done.
\end{proof}

\section{Second order estimate} \label{sectionsecond}

In this section we continue the proof of Theorem \ref{mainthm} by establishing an
 estimate on the metric $\tilde{g}$ in terms of the gradient of $u$.  The setup is the same as in the previous section, and in particular, $\tilde{g}$ is a Hermitian metric given by
 \begin{equation} \label{t1}
\tilde{g}_{i\ov{j}} = h_{i\ov{j}} + \frac{1}{n-1} \left( (\Delta u)g_{i\ov{j}} -u_{i\ov{j}} \right),
\end{equation}
which solves
\begin{equation} \label{es3}
\det \tilde{g} = e^{F+b} \det g.
\end{equation}

We prove:

\begin{theorem} \label{theoremC2}
There exists a uniform constant $C$ depending only on $(M, \omega_0), \omega$ and bounds for $F$ such that
$$\emph{tr}_{g}{\tilde{g}} \le C( \sup_M |\nabla u|^2_g + 1).$$
\end{theorem}

Noting that $\tr{g}{\tilde{g}} = \tr{g}{h} + \Delta u$, we
  define, as in \cite{TW4},
\begin{equation} \label{defneta}
\eta_{i\ov{j}} =u_{i\ov{j}} +  (\tr{g}{h}) g_{i\ov{j}} - (n-1) h_{i\ov{j}}    = (\tr{g}{\tilde{g}})  g_{i\ov{j}} - (n-1) \tilde{g}_{i\ov{j}}.
\end{equation}
For $x\in M$ and $\xi$ a unit vector with respect to $g$, we define
 $H(x, \xi)$  by
$$H(x, \xi) = \log (\eta_{i\ov{j}} \xi^i \ov{\xi^j})+ c \log(g^{p\ov{q}}\eta_{i\ov{q}} \eta_{p\ov{j}}\xi^i\ov{\xi^j}) +   \varphi (|\nabla u|^2_g) + \psi(u),$$
where $\varphi, \psi$  are given by
\begin{equation}
\begin{split}
\varphi(s) = {} & - \frac{1}{2} \log \left( 1- \frac{s}{2K} \right), \quad \textrm{for } 0 \le s \le K-1, \\
\psi(t) = {} & - A \log \left( 1+ \frac{t}{2L} \right), \quad \textrm{for } -L +1 \le t \le 0,
\end{split}
\end{equation}
for
$$K: = \sup_M | \nabla u|^2_g+1,\ L: = \sup_M|u| +1, \ A: = 2L(C_1 +1)$$
and $C_1$ is   to be determined later.  The constant $c>0$ is a small constant, depending only on $n$, which will also be determined later.
 The quantities $\varphi$ and $\psi$ are identical to those in the computation of  Hou-Ma-Wu \cite{HMW}.

Recall that we have normalized $u$ so that $\sup_M u=0$. $L$ is uniformly bounded. The quantities $\varphi(|\nabla u|^2_g)$ and $\psi(u)$ are both uniformly bounded.  For later use note that,
as in \cite{HMW}, we have
\begin{equation} \label{hmwc1}
\begin{split}
&\frac{1}{2K} \ge  \varphi' \ge \frac{1}{4K} >0, \quad  \varphi'' =  2(\varphi')^2 >0
\end{split}
\end{equation}
and
\begin{equation} \label{psidp}
\begin{split}
&\frac{A}{L} \ge  - \psi' \ge \frac{A}{2L} = C_1+1, \quad \psi'' \ge  \frac{2\ve}{1-\ve} (\psi')^2, \quad \textrm{for all } \ve \le \frac{1}{2A+1},
\end{split}
\end{equation}
whenever $\vp$ and $\psi$ are evaluated at $|\nabla u|^2_g$ and $u$ respectively.

Note that the difference between $H$ here and the quantity considered in \cite{TW4} is that we have added a small multiple of the quantity $$\log (g^{p\ov{q}} \eta_{i\ov{q}} \eta_{p\ov{j}} \xi^i \ov{\xi^j}).$$  This is important in what follows, since after applying the linearized operator to our $H$ we obtain additional positive terms (see (\ref{LQ}), below) which
are needed to  bound torsion terms that did not arise in \cite{TW4}.

We restrict the function $H$ to the compact set $W$ in the $g$-unit tangent bundle of $M$ where $\eta_{i\ov{j}} \xi^i \ov{\xi^j} \geq 0$, defining $H=-\infty$ whenever
$\eta_{i\ov{j}} \xi^i \ov{\xi^j} =0$. Note that $\eta_{i\ov{j}} \xi^i \ov{\xi^j} >0$ implies $g^{p\ov{q}}\eta_{i\ov{q}} \eta_{p\ov{j}}\xi^i\ov{\xi^j} >0.$ This way, $H$ is 
upper semi-continuous on $W$ and hence $H$ achieves a maximum at some point $(x_0, \xi_0)$ where $\eta_{i\ov{j}}(x_0) \xi_0^i \ov{\xi_0^j} >0$.

Choose a holomorphic coordinate system $z^1, \ldots, z^n$ centered at $x_0$ such that at $x_0$,
$$g_{i\ov{j}} = \delta_{ij}, \quad \eta_{i\ov{j}} = \delta_{ij} \eta_{i\ov{i}}, \quad  \eta_{1\ov{1}} \ge \eta_{2\ov{2}} \ge \cdots \ge \eta_{n \ov{n}}.$$
From (\ref{defneta}), $(\tilde{g}_{i\ov{j}})$ is also diagonal at this point.  Writing $\lambda_i = \tilde{g}_{i\ov{i}}$ we have at $x_0$,
\begin{equation} \label{etalambda}
\eta_{i\ov{i}} = \sum_{j=1}^n \lambda_j - (n-1) \lambda_i,
\end{equation}
and in particular
$0<\lambda_1 \le \lambda_2 \le \cdots \le \lambda_n.$
The three quantities $\lambda_n$, $\eta_{1\ov{1}}$ and $\tr{\omega}{\tilde{\omega}}$ are uniformly equivalent:
\begin{equation} \label{ie}
\frac{1}{n} \tr{\omega}{\tilde{\omega}} \le \lambda_n \le \eta_{1\ov{1}} \le (n-1) \lambda_n \le (n-1) \tr{\omega}{\tilde{\omega}}.
\end{equation}
Observe that $\eta_{n\ov{n}}$ could be negative, and it is possible in general  that $ - \eta_{n\ov{n}} > \eta_{1\ov{1}}$. This only happens for $n\geq 4$, since it is easy to check that
when $n=3$ we have $\eta_{1\ov{1}}>-\eta_{3\ov{3}}$.  However, in any case we have
\begin{equation} \label{eb}
| \eta_{n\ov{n}}| \le (n-2) \eta_{1\ov{1}}.
\end{equation}
At $x_0$ we have
$$\log (\eta_{i\ov{j}} \xi^i \ov{\xi^j}) + c \log (g^{p\ov{q}}\eta_{i\ov{q}} \eta_{p\ov{j}}\xi^i\ov{\xi^j})  =  \log \left( \sum_i \eta_{i\ov{i}} | \xi^i|^2 \right) + c \log \left( \sum_i \eta_{i\ov{i}}^2|\xi^i|^2\right).$$
Using \eqref{eb}, the reader can check that if $c>0$ is chosen sufficiently small, depending only on $n$, then this quantity is maximized by $\xi_0=\de/\de z^1$. 
 In fact, one can take any positive $c$ with $c<1/(n-3)$ for $n>3$ and any $c>0$ for $n=3$.

We extend $\xi_0$ to a locally defined smooth unit vector field
\begin{equation} \label{xid}
\xi_0 = g_{1\ov{1}}^{-1/2} \frac{\partial}{\partial z^1}.
\end{equation}
Define a new quantity,  in a neighborhood of $x_0$, by
$$Q(x) = H(x, \xi_0) =\log (g_{1\ov{1}}^{-1} \eta_{1\ov{1}})+c \log(g_{1\ov{1}}^{-1}g^{i\ov{j}} \eta_{1\ov{j}}\eta_{i\ov{1}} )  + \varphi (|\nabla u|^2_g) + \psi(u),$$
which achieves a maximum at $x_0$.
Note that at $x_0$ we have
$$g_{1\ov{1}}^{-1}g^{i\ov{j}} \eta_{1\ov{j}}\eta_{i\ov{1}}=(\eta_{1\ov{1}})^2.$$
Our goal is to prove that, at $x_0$,
\begin{equation} \label{gol}
|\eta_{1\ov{1}}| \le CK,
\end{equation}
for a uniform constant $C$.  This will prove the theorem: at any given point $x\in M$, we choose coordinates are before, so that
 $\frac{1}{n}\tr{\omega}{\tilde{\omega}}(x)$ is bounded from above by $\eta_{1\ov{1}}(x)$. But this is in turn bounded
 above by 
 $$\sup_{\xi \in W}\left\{ (\eta_{k\ov{\ell}} \xi^k \ov{\xi^{\ell}})^{1/(1+2c)}   \left( g^{p\ov{q}}\eta_{i\ov{q}} \eta_{p\ov{j}}\xi^i\ov{\xi^j} \right)^{c/(1+2c)} \right\},$$ and so
 \[
 \begin{split}
 \sup_M \tr{\omega}{\tilde{\omega}} \le {} & n   \sup_{\xi \in W} \left\{ (\eta_{k\ov{\ell}} \xi^k \ov{\xi^{\ell}})^{1/(1+2c)}   \left( g^{p\ov{q}}\eta_{i\ov{q}} \eta_{p\ov{j}}\xi^i\ov{\xi^j} \right)^{c/(1+2c)} \right\} \\
 \le {} &
Ce^{Q(x_0)/(1+2c)} \le C'K= C'(\sup_M |\nabla u|^2_g + 1),
\end{split}
\]
as required.
We may and do assume, without loss of generality, that $|\eta_{1\ov{1}}|>>1$ at $x_0$.

Define a linear operator $L$ on functions by $L(v)= \Theta^{i\ov{j}} \partial_i \ov{\partial}_j v$ where
 $$\Theta^{i\ov{j}} = \frac{1}{n-1} ( (\tr{\tilde{g}}{g})g^{i\ov{j}} - \tilde{g}^{i\ov{j}} )>0.$$
At $x_0$, $\Theta^{i\ov{j}}$ is diagonal, and we have
\begin{equation} \label{sumT}
\sum_i \Theta^{i\ov{i}} = \tr{\tilde{g}}{g}.
\end{equation}
We will apply the maximum principle to $Q$ by computing $L(Q)$.
We use covariant derivatives with respect to the Hermitian metric  $g$, which we denote by subscripts.  First, at $x_0$, we have from (\ref{xid}), dropping the zero subscript in $\xi_0$, 
\begin{equation} \label{xieqn}
\xi^1_i + \ov{\xi^1_{\ov{i}}}=0.
\end{equation}
Using this, compute
\begin{equation} \label{Qi}
0 = Q_i =(1+2c) \frac{\eta_{1 \ov{1} i}}{\eta_{1\ov{1}}} + \varphi'  \left( \sum_p u_p u_{\ov{p}i} +  \sum_p u_{pi} u_{\ov{p}} \right) + \psi'  u_i.
\end{equation}
Next
\begin{equation}
\begin{split} \label{LQold}
L(Q) =&(1+2c)\sum_i\frac{\Theta^{i\ov{i}}   \eta_{1\ov{1} i\ov{i}}}{\eta_{1\ov{1}}}
 + c \sum_i\sum_{p\neq 1}\frac{\Theta^{i\ov{i}}  |\eta_{p\ov{1} i}|^2}{(\eta_{1\ov{1}})^2}+ c\sum_i\sum_{p\neq 1}\frac{\Theta^{i\ov{i}}  | \eta_{1\ov{p} i}|^2}{(\eta_{1\ov{1}})^2} \\
 &
 - (1+2c) \sum_{i} \frac{ \Theta^{i\ov{i}}|\eta_{1\ov{1} i}|^2}{(\eta_{1\ov{1}})^2}
+ (*) + \psi' \sum_i \Theta^{i\ov{i}} u_{i\ov{i}}
+ \psi'' \sum_i \Theta^{i\ov{i}} |u_i|^2 \\
&  + \varphi'' \sum_i \Theta^{i\ov{i}} \left| \sum_p u_p u_{\ov{p} i} + \sum_p u_{pi} u_{\ov{p}} \right|^2
+ \varphi' \sum_{i,p} \Theta^{i\ov{i}} \left( |u_{p \ov{i}}|^2  +  |u_{pi}|^2 \right) \\
&+ \varphi' \sum_{i,p} \Theta^{i\ov{i}} (u_{pi\ov{i}} u_{\ov{p}} + u_{\ov{p} i \ov{i}} u_p),
\end{split}
\end{equation}
where 
\begin{equation}
\begin{split}
(*) = {}& 2(1+c) \sum_{p \neq 1} \sum_i \frac{\Theta^{i\ov{i}}}{\eta_{1\ov{1}}}\textrm{Re} (\eta_{p\ov{1}i} \xi^p_{\ov{i}} + \eta_{1\ov{p}i} \ov{\xi^p_i}) \\
& + 2c \sum_{p\neq 1} \sum_i \frac{\Theta^{i\ov{i}}}{(\eta_{1\ov{1}})^2} \textrm{Re}(   \eta_{p\ov{1} i} \eta_{p\ov{p}} \xi^p_{\ov{i}} + \eta_{1\ov{p}i} \eta_{p\ov{p}} \ov{\xi^p_i}) \\
& + (1+c) \sum_i \Theta^{i\ov{i}} (\xi^1_{i\ov{i}} + \ov{\xi^1_{\ov{i} i}} ) + \sum_{p,i} \frac{\Theta^{i\ov{i}} \eta_{p\ov{p}}}{\eta_{1\ov{1}}} (| \xi^p_i|^2 + |\xi^p_{\ov{i}}|^2) \\
& + c \sum_{p,i} \frac{\Theta^{i\ov{i}} (\eta_{p\ov{p}})^2}{(\eta_{1\ov{1}})^2} (| \xi^p_i|^2 + |\xi^p_{\ov{i}}|^2).
\end{split}
\end{equation}
We have used again the equation (\ref{xieqn}), which is why the first two $p$ summations in $(*)$  do not include $p=1$.
Observe that the second and third terms on the first line of the expression for $L(Q)$ give new positive terms that did not appear in \cite{TW4}.  We will use half of these terms to bound $(*)$, and the other half later.  We obtain:
\begin{equation}
\begin{split} \label{LQ}
L(Q) \ge &(1+2c)\sum_i\frac{\Theta^{i\ov{i}}   \eta_{1\ov{1} i\ov{i}}}{\eta_{1\ov{1}}}
 + \frac{c}{2} \sum_i\sum_{p\neq 1}\frac{\Theta^{i\ov{i}}  |\eta_{p\ov{1} i}|^2}{(\eta_{1\ov{1}})^2}+ \frac{c}{2}\sum_i\sum_{p\neq 1}\frac{\Theta^{i\ov{i}}  | \eta_{1\ov{p} i}|^2}{(\eta_{1\ov{1}})^2} \\
 &
 - (1+2c) \sum_{i} \frac{ \Theta^{i\ov{i}}|\eta_{1\ov{1} i}|^2}{(\eta_{1\ov{1}})^2}
 + \psi' \sum_i \Theta^{i\ov{i}} u_{i\ov{i}}
+ \psi'' \sum_i \Theta^{i\ov{i}} |u_i|^2 \\
&  + \varphi'' \sum_i \Theta^{i\ov{i}} \left| \sum_p u_p u_{\ov{p} i} + \sum_p u_{pi} u_{\ov{p}} \right|^2
+ \varphi' \sum_{i,p} \Theta^{i\ov{i}} \left( |u_{p \ov{i}}|^2  +  |u_{pi}|^2 \right) \\
&+ \varphi' \sum_{i,p} \Theta^{i\ov{i}} (u_{pi\ov{i}} u_{\ov{p}} + u_{\ov{p} i \ov{i}} u_p) - C \tr{\tilde{g}}{g},
\end{split}
\end{equation}
for a uniform $C$.

Differentiating (\ref{es3}), we obtain
\begin{equation} \label{Fell}
\tilde{g}^{i\ov{j}} \nabla_{\ell} \tilde{g}_{i\ov{j}} = F_{\ell}, \quad \tilde{g}^{i\ov{j}} \nabla_{\ov{m}} \nabla_{\ell} \tilde{g}_{i\ov{j}} - \tilde{g}^{i\ov{q}} \tilde{g}^{p\ov{j}} \nabla_{\ov{m}} \tilde{g}_{p\ov{q}} \nabla_{\ell} \tilde{g}_{i\ov{j}} = F_{\ell \ov{m}}.
\end{equation}
In particular, we obtain at $x_0$
\begin{equation}\label{atx}
\begin{split}
\lefteqn{\sum_i\Theta^{i\ov{i}} u_{i \ov{i} 1 \ov{1}} + \sum_i \tilde{g}^{i\ov{i}} h_{i\ov{i} 1 \ov{1}}}\\ & - \frac{1}{(n-1)^2} \sum_{i,j}\tilde{g}^{i\ov{i}} \tilde{g}^{j\ov{j}} ( g_{j\ov{i}} \sum_a u_{a\ov{a} \ov{1}} - u_{j\ov{i} \ov{1}}  + \hat{h}_{j\ov{i}\ov{1}} ) (g_{i\ov{j}} \sum_b u_{b\ov{b} 1} - u_{i\ov{j}1} + \hat{h}_{i\ov{j} 1}) =F_{1 \ov{1}},
\end{split}
\end{equation}
where we are writing $\hat{h}_{i\ov{j}} = (n-1) h_{i\ov{j}}$.

We have the following commutation formulae:
\begin{equation}  \label{comm}
\begin{split}
u_{i\ov{j} \ell} = {} & u_{i\ell \ov{j}} - u_p R_{\ell \ov{j} i}^{\ \ \ \, p}, \ \,
u_{p\ov{j} \ov{m}} =  u_{p \ov{m} \ov{j}} - \ov{T^q_{mj}} u_{p\ov{q}}, \ \,
u_{i\ov{q} \ell} =    u_{\ell \ov{q} i} - T^p_{\ell i} u_{p\ov{q}} \\
u_{i\ov{j} \ell \ov{m}} = {} & u_{\ell \ov{m} i\ov{j}} + u_{p\ov{j}} R_{\ell \ov{m} i}^{\ \ \ \  p}   -
 u_{p\ov{m}} R_{i \ov{j} \ell}^{\ \ \ \, p}  - T^p_{\ell i} u_{p\ov{m} \ov{j}}
  - \ov{T_{mj}^q} u_{\ell\ov{q} i} - T^p_{i \ell} \ov{T^q_{mj}} u_{p\ov{q}}.
\end{split}
\end{equation}
Indeed, these follow from the standard formulae (using conventions of \cite{TW3}):
$$\nabla_i a_{\ell} = \partial_i a_{\ell} - \Gamma^p_{i\ell} a_p, \ \Gamma^k_{ij} = g^{k\ov{q}} \partial_i g_{j \ov{q}}, \ T^k_{ij} = \Gamma^k_{ij} - \Gamma^k_{ji}, \ R_{\ell \ov{m}i}^{\ \ \ \ p} = - \partial_{\ov{m}} \Gamma^p_{\ell i},$$
and
\begin{equation} \label{cf}
[\nabla_i, \nabla_{\ov{j}}] a_{\ell}= - R_{i \ov{j} \ell}^{\ \ \ \, p} a_p, \ [\nabla_i, \nabla_{\ov{j}}] a_{\ov{m}} = R_{i\ov{j} \ \, \ov{m}}^{\ \ \, \ov{q}} \ov{a_q}.
\end{equation}

Then from \eqref{atx} and (\ref{comm}),
\begin{equation}
\begin{split} \label{F112}
\lefteqn{\sum_i \Theta^{i\ov{i}} u_{ 1 \ov{1} i\ov{i} }+ \sum_i \tilde{g}^{i\ov{i}} h_{i\ov{i} 1 \ov{1}} + \sum_{i} \Theta^{i\ov{i}} \left(  u_{a\ov{i}} R_{1 \ov{1} i}^{ \ \ \ \, a} -  u_{a\ov{1}} R_{i \ov{i} 1}^{ \ \ \ a} \right) } \\
& + \sum_i \Theta^{i\ov{i}} \left( -  T^p_{1 i} u_{p\ov{1} \ov{i}} -  \ov{T^p_{1 i}} u_{1\ov{p}i} -  T^a_{i1} \ov{T^b_{1 i}} u_{a\ov{b}} \right) \\
&  - \frac{1}{(n-1)^2} \sum_{i,j} \tilde{g}^{i\ov{i}} \tilde{g}^{j\ov{j}} ( g_{j\ov{i}} \sum_a u_{a\ov{a} \ov{1}} - u_{j\ov{i} \ov{1}} + \hat{h}_{ji\ov{1}} ) (g_{i\ov{j}} \sum_b u_{b\ov{b} 1} - u_{i\ov{j}1} + \hat{h}_{i\ov{j} 1} ) =F_{1 \ov{1}}.
\end{split}
\end{equation}
From (\ref{defneta}),
\begin{equation} \label{ueta}
u_{1\ov{1} i\ov{i}} = \eta_{1\ov{1}i\ov{i}} + \hat{h}_{1 \ov{1} i\ov{i}} - (\tr{g}{h})_{i\ov{i}}.
\end{equation}

From  (\ref{LQ}), (\ref{F112}) and (\ref{ueta}) and using the fact that at $x_0$ we have $L(Q)\le 0$, we obtain
\begin{equation} \label{firstcombine}
\begin{split}
0 \ge {} &(1+2c)\sum_{i,j} \frac{ \tilde{g}^{i\ov{i}} \tilde{g}^{j\ov{j}} ( g_{j\ov{i}} \sum_a u_{a\ov{a} \ov{1}} - u_{j\ov{i} \ov{1}} + \hat{h}_{j\ov{i} \ov{1}} ) (g_{i\ov{j}} \sum_b u_{b\ov{b} 1} - u_{i\ov{j}1} + \hat{h}_{i\ov{j} 1} )}{{(n-1)^2 \eta_{1\ov{1}}}} \\
& +\frac{c}{2}\sum_i\sum_{p\neq 1}\frac{\Theta^{i\ov{i}}  |\eta_{p\ov{1} i}|^2}{(\eta_{1\ov{1}})^2}+ \frac{c}{2}\sum_i\sum_{p\neq 1}\frac{\Theta^{i\ov{i}}  | \eta_{1\ov{p} i}|^2}{(\eta_{1\ov{1}})^2}
 - (1+2c)\sum_{i} \frac{ \Theta^{i\ov{i}}|\eta_{1\ov{1} i}|^2}{(\eta_{1\ov{1}})^2}\\
& - C \tr{\tilde{g}}{g} + \frac{1+2c}{\eta_{1\ov{1}}} \left(    F_{1\ov{1}} - \sum_{i} \Theta^{i\ov{i}} \left(  u_{a\ov{i}} R_{1 \ov{1} i}^{ \ \ \ \, a} -   u_{a\ov{1}} R_{i \ov{i} 1}^{ \ \ \  a} \right)
  \right. \\
 & \left.  - \sum_i \tilde{g}^{i\ov{i}} h_{i\ov{i} 1 \ov{1}} - \sum_i \Theta^{i\ov{i}}\left( \hat{h}_{1 \ov{1} i \ov{i} } - (\tr{g}{h})_{i\ov{i}}  \right) \right) \\
&  +  \frac{1+2c}{\eta_{1\ov{1}}} \left( \sum_{i} \Theta^{i\ov{i}} \left(    T^p_{1 i} u_{p\ov{1} \ov{i}} + \ov{T^p_{1 i}} u_{1\ov{p}i} +   T^a_{i1} \ov{T^b_{1 i}} u_{a\ov{b}} \right)  \right) \\
& + \psi' \sum_i \Theta^{i\ov{i}} u_{i\ov{i}}
+ \psi'' \sum_i \Theta^{i\ov{i}} |u_i|^2 + \varphi'' \sum_i \Theta^{i\ov{i}} \left| \sum_p u_p u_{\ov{p} i} + \sum_p u_{pi} u_{\ov{p}} \right|^2 \\
&+ \varphi' \sum_{i,p} \Theta^{i\ov{i}} \left( |u_{p \ov{i}}|^2  +  |u_{pi}|^2 \right) + \varphi' \sum_{i,p} \Theta^{i\ov{i}} (u_{pi\ov{i}} u_{\ov{p}} + u_{\ov{p} i \ov{i}} u_p). \\
=:{} & (1) + (2) + (3) + (4) + (5) + (6) + (7),
 \end{split}
\end{equation}
where the numbers (1) - (7) refer to the lines in the expression above.

We now make the observation that we may and do assume that, at the point $x_0$,
\begin{equation} \label{apprbds2}
\begin{split}
& | u_{i\ov{j}}| \le  \ 2|\eta_{1\ov{1}}|,\quad \textrm{for all } i,j.
\end{split}
\end{equation}
Indeed, since
 our goal is to prove (\ref{gol}) we may assume without loss of generality that $|\eta_{1 \ov{1}}|$ is large. Then (\ref{apprbds2}) follows   immediately from \eqref{defneta}.

We will now deal with each line of (\ref{firstcombine}) in turn, starting with the easiest.

\bigskip
\noindent
{\bf Lines (3) and (4) of (\ref{firstcombine}).} \ From  (\ref{sumT}) and \eqref{apprbds2}, we immediately have a lower bound for the third and fourth lines of (\ref{firstcombine}):
\begin{equation} \label{4linelb}
\begin{split}
(3)+(4) \ge {}  & - C \tr{\tilde{g}}{g} -C.
\end{split}
\end{equation}

\bigskip
\noindent
{\bf Line (6) of (\ref{firstcombine}).} \ The second and third terms in this line are nonnegative (since $\psi'', \varphi''>0$) and we will make use of them later.  Note that by taking trace with respect to $\tilde{g}$ of (\ref{t1}) we obtain
\begin{equation} \label{Thetauii}
\psi' \sum_i \Theta^{i\ov{i}} u_{i\ov{i}}  = (-\psi') ( \tr{\tilde{g}}{h}-n).
\end{equation}
On the other hand recall that from (\ref{psidp}),
$$\frac{A}{L} \ge  - \psi' \ge \frac{A}{2L} = C_1+1,$$
where $C_1$ is still to be determined.   So we have
\begin{equation}
\begin{split}
\psi' \sum_i \Theta^{i\ov{i}} u_{i\ov{i}}
\ge {} & -2n(C_1+1) + (C_1+1) \tr{\tilde{g}}{h}
\end{split}
\end{equation}
and hence
\begin{equation}
\begin{split}
(6) \ge {} &  -2n(C_1+1) + (C_1+1) \tr{\tilde{g}}{h}
 + \psi'' \sum_i \Theta^{i\ov{i}} |u_i|^2
 \\ & + \varphi'' \sum_i \Theta^{i\ov{i}} \left| \sum_p u_p u_{\ov{p} i} + \sum_p u_{pi} u_{\ov{p}} \right|^2.
\end{split}
\end{equation}

\bigskip
\noindent
{\bf Line (7) of (\ref{firstcombine}).} \ The first  term of line (7) is nonnegative since $\varphi'>0$.   For the second term we argue as follows.
From (\ref{Fell}), we have at $x_0$,
$$\sum_i \Theta^{i\ov{i}} u_{i\ov{i} p} = F_{p}- \sum_i \tilde{g}^{i\ov{i}} h_{i\ov{i}p}.$$
Using (\ref{comm}) we have at $x_0$,
\begin{equation}
\begin{split}
\lefteqn{ \varphi' \sum_{i,p} \Theta^{i\ov{i}} (u_{pi\ov{i}} u_{\ov{p}} + u_{\ov{p} i\ov{i}} u_p )  } \\={} & \varphi' \sum_p ( F_p u_{\ov{p}} + F_{\ov{p}} u_p)
 + \varphi' \sum_{i,p,q} \Theta^{i\ov{i}} u_q u_{\ov{p}} R_{i\ov{i} p}^{\ \ \ q} \\ & - \varphi' \sum_{i,p} \tilde{g}^{i\ov{i}} (h_{i\ov{i}p} u_{\ov{p}} + h_{i\ov{i} \ov{p}} u_p )
 + 2\varphi' \mathrm{Re}\sum_{i,p,q} \Theta^{i\ov{i}} u_{\ov{p}} u_{q\ov{i}} T^q_{pi}.
 \end{split}
\end{equation}
Making use of (\ref{hmwc1}) we have the estimate
\begin{equation} \label{phiplb}
\varphi' \sum_{i,p} \Theta^{i\ov{i}} (u_{pi\ov{i}} u_{\ov{p}} + u_{\ov{p} i\ov{i}} u_p )   \ge - C - C \tr{\tilde{g}}{g} - \frac{1}{10} \varphi' \sum_{i,p} \Theta^{i\ov{i}}(|u_{pi}|^2+ |u_{p\ov{i}}|^2).
\end{equation}
Hence
\begin{equation}\label{l7b}
(7) \ge \frac{9}{10} \varphi' \sum_{i,p} \Theta^{i\ov{i}} (|u_{p\ov{i}}|^2 + |u_{pi}|^2) - C - C \tr{\tilde{g}}{g}.
\end{equation}

\bigskip
\noindent
{\bf Line (5) of (\ref{firstcombine}).} \ Here is where torsion terms appear that need to be controlled using the first and second terms of line (2) of (\ref{firstcombine}).
We first bound
\[
\begin{split}
\frac{1+2c}{\eta_{1\ov{1}}} \sum_{i,p} \Theta^{i\ov{i}} \left(  T^p_{1 i} u_{p\ov{1} \ov{i}} + \ov{T^p_{1 i}} u_{1\ov{p}i} \right) = \frac{2(1+2c)}{\eta_{1\ov{1}}} \sum_{i,p} \Theta^{i\ov{i}} \textrm{Re} \left( \ov{T^p_{1 i}} u_{1\ov{p}i} \right).
\end{split}
\]
We have
$$u_{1\ov{p}i} = \eta_{1\ov{p} i} - (\tr{g}{h})_i g_{1 \ov{p}} + \hat{h}_{1 \ov{p} i}.$$
Hence
\begin{equation} \label{line51}
\begin{split}
\lefteqn{
\frac{2(1+2c)}{\eta_{1\ov{1}}} \sum_{i,p} \Theta^{i\ov{i}} \textrm{Re} \left( \ov{T^p_{1 i}} u_{1\ov{p}i} \right) }  \\
\ge {} & \frac{2(1+2c)}{\eta_{1\ov{1}}} \sum_{i,p} \Theta^{i\ov{i}} \textrm{Re} \left( \ov{T^p_{1 i}} \eta_{1 \ov{p}i} \right) - C \tr{\tilde{g}}{g}.
\end{split}
\end{equation}
Notice that if we consider only the summands $p \neq 1$, we have
\begin{equation} \label{line52}
\frac{2(1+2c)}{\eta_{1\ov{1}}} \sum_{i} \sum_{p\neq 1} \Theta^{i\ov{i}} \textrm{Re} \left( \ov{T^p_{1 i}} \eta_{1\ov{p}i} \right) \ge - \frac{c}{4} \sum_i \sum_{p\neq 1} \Theta^{i\ov{i}} \frac{| \eta_{1\ov{p}i}|^2}{(\eta_{1\ov{1}})^2} - C \tr{\tilde{g}}{g},
\end{equation}
and observe that the first term on the right hand side of this inequality can be controlled by the second term on line (2) of (\ref{firstcombine}).

The term when $p=1$ can be written as
\begin{equation}\label{line53}
\frac{2(1+2c)}{\eta_{1\ov{1}}} \sum_{i}  \Theta^{i\ov{i}} \textrm{Re} \left( \ov{T^1_{1 i}} \eta_{1\ov{1}i} \right)  = \frac{2(1+2c)}{\eta_{1\ov{1}}} \sum_{i\neq1}  \Theta^{i\ov{i}} \textrm{Re} \left( \ov{T^1_{1 i}} \eta_{1\ov{1}i} \right),
\end{equation}
using the skew-symmetry of torsion.

The third term of line (5) of (\ref{firstcombine}) can be easily bounded by
\begin{equation}\label{line54}
\left|\frac{1+2c}{\eta_{1\ov{1}}}  \sum_{i} \Theta^{i\ov{i}}  T^a_{i1} \ov{T^b_{1 i}} u_{a\ov{b}} \right|\leq C \tr{\tilde{g}}{g}.
\end{equation}
Combining (\ref{line51}), (\ref{line52}), (\ref{line53}) and \eqref{line54}, we obtain
\begin{equation} \label{mystery1}
\begin{split}
(5) \ge  {} &\frac{2(1+2c)}{\eta_{1\ov{1}}} \sum_{i\neq1} \Theta^{i\ov{i}} \textrm{Re} \left( \ov{T^1_{1 i}} \eta_{1\ov{1}i} \right) - \frac{c}{4} \sum_i \sum_{p\neq 1} \Theta^{i\ov{i}} \frac{| \eta_{1\ov{p}i}|^2}{(\eta_{1\ov{1}})^2}  - C \tr{\tilde{g}}{g}.
\end{split}
\end{equation}

For the moment we leave lines $(1)$ and $(2)$ of \eqref{firstcombine} as they are, and combining all the above lower bounds for $(1),\dots,(7)$, we obtain
\begin{equation}\label{secondcombine}
\begin{split}
0\geq {} &(1+2c)\sum_{i,j} \frac{ \tilde{g}^{i\ov{i}} \tilde{g}^{j\ov{j}} ( g_{j\ov{i}} \sum_a u_{a\ov{a} \ov{1}} - u_{j\ov{i} \ov{1}} + \hat{h}_{j\ov{i} \ov{1}}) (g_{i\ov{j}} \sum_b u_{b\ov{b} 1} - u_{i\ov{j}1} + \hat{h}_{i\ov{j}1} )}{{(n-1)^2 \eta_{1\ov{1}}}} \\
& + \frac{c}{2} \sum_i\sum_{p\neq 1}\frac{\Theta^{i\ov{i}}  |\eta_{p\ov{1} i}|^2}{(\eta_{1\ov{1}})^2}+ \frac{c}{4}\sum_i\sum_{p\neq 1}\frac{\Theta^{i\ov{i}}  | \eta_{1\ov{p} i}|^2}{(\eta_{1\ov{1}})^2}
 - (1+2c)\sum_{i} \frac{ \Theta^{i\ov{i}}|\eta_{1\ov{1} i}|^2}{(\eta_{1\ov{1}})^2}\\
& + \psi'' \sum_i \Theta^{i\ov{i}} |u_i|^2+ \varphi'' \sum_i \Theta^{i\ov{i}} \left| \sum_p u_p u_{\ov{p} i} + \sum_p u_{pi} u_{\ov{p}} \right|^2\\
& +  \frac{9}{10} \varphi' \sum_{i,p} \Theta^{i\ov{i}} ( |u_{pi}|^2 + |u_{p \ov{i}}|^2)- C_0 \tr{\tilde{g}}{g}-C\\
& {}  +\frac{2(1+2c)}{\eta_{1\ov{1}}} \sum_{i\neq1}  \Theta^{i\ov{i}} \textrm{Re} \left( \ov{T^1_{1 i}} \eta_{1\ov{1}i} \right)
   + (C_1+1) \tr{\tilde{g}}{h},
\end{split}
\end{equation}
where $C_0$ is a uniform constant (independent of the value of $C_1$).
Let $C_2$ be a uniform constant such that
\begin{equation}\label{torsionbound}
\sup_M|T|^2_g\leq C_2.
\end{equation}
We now pick $C_1$ sufficiently large (and uniformly bounded) so that
$$(C_1 +1) \tr{\tilde{g}}{h} \ge (C_0 +2) \tr{\tilde{g}}{g}$$
 and define
$$\delta=\min\left\{\frac{1}{1+4L(C_1+1)},\frac{1}{16(1+2c)^2C_2}\right\},$$
which is a small uniform constant.

We first consider:

\bigskip
\noindent
{\bf Case 1:  $\lambda_2 \leq (1-\delta) \lambda_n$.}

\bigskip
In this case, we simply throw away several nonnegative terms in \eqref{secondcombine} and get
\begin{equation}\label{secondcombine2}
\begin{split}
0\geq {} & - (1+2c)\sum_{i} \frac{ \Theta^{i\ov{i}}|\eta_{1\ov{1} i}|^2}{(\eta_{1\ov{1}})^2}+ \varphi'' \sum_i \Theta^{i\ov{i}} \left| \sum_p u_p u_{\ov{p} i} + \sum_p u_{pi} u_{\ov{p}} \right|^2- C_0 \tr{\tilde{g}}{g}-C\\
& {}  +\frac{2(1+2c)}{\eta_{1\ov{1}}} \sum_{i\neq1}  \Theta^{i\ov{i}} \textrm{Re} \left( \ov{T^1_{1 i}} \eta_{1\ov{1}i} \right)
 +\frac{9}{10} \varphi' \sum_{i,p} \Theta^{i\ov{i}} u_{i \ov{i}}^2 + (C_1+1) \tr{\tilde{g}}{h}.
\end{split}
\end{equation}
We then bound
\[\begin{split}
\left|\frac{2(1+2c)}{\eta_{1\ov{1}}} \sum_{i\neq1}  \Theta^{i\ov{i}} \textrm{Re} \left( \ov{T^1_{1 i}} \eta_{1\ov{1}i} \right)\right|
&\leq 2c\sum_i \Theta^{i\ov{i}} \frac{| \eta_{1\ov{1}i}|^2}{(\eta_{1\ov{1}})^2}+C\tr{\tilde{g}}{g}.
\end{split}\]
Hence
\begin{equation}\label{final}
\begin{split}
0\geq {} & - (1+2c)^2\sum_{i} \frac{ \Theta^{i\ov{i}}|\eta_{1\ov{1} i}|^2}{(\eta_{1\ov{1}})^2}+ \varphi'' \sum_i \Theta^{i\ov{i}} \left| \sum_p u_p u_{\ov{p} i} + \sum_p u_{pi} u_{\ov{p}} \right|^2\\
& +\frac{9}{10} \varphi' \sum_{i,p} \Theta^{i\ov{i}} u_{i \ov{i}}^2 -C \tr{\tilde{g}}{g}-C.
\end{split}
\end{equation}
From (\ref{Qi}) and (\ref{sumT}) we have
\begin{equation}
\begin{split} \label{ft}
- (1+2c)^2\sum_i \frac{\Theta^{i\ov{i}} | \eta_{1\ov{1}i}|^2}{(\eta_{1\ov{1}})^2} =  {} & - \sum_i \Theta^{i\ov{i}} \left| \varphi'  \left( \sum_p u_p u_{\ov{p}i} +  \sum_p u_{pi} u_{\ov{p}} \right) + \psi'  u_i \right|^2 \\
\ge {} &  - 2 (\varphi')^2 \sum_i \Theta^{i\ov{i}} \left| \sum_p u_p u_{\ov{p}i} +  \sum_p u_{pi} u_{\ov{p}} \right|^2 \\
&- 8 (C_1+1)^2 K\tr{\tilde{g}}{g},
\end{split}
\end{equation}
where we have used $| \psi'| \le A/L = 2(C_1+1)$.  Then, using (\ref{hmwc1}) and (\ref{psidp}), we have from \eqref{final},
\begin{equation}
\begin{split}
0 \ge {} & \frac{1}{8K} \sum_i \Theta^{i\ov{i}} u_{i\ov{i}}^2 - CK \tr{\tilde{g}}{g} -C\tr{\tilde{g}}{g},
\end{split}
\end{equation}
giving
\begin{equation}
\begin{split}
 \frac{1}{8K} \Theta^{n\ov{n}} u_{n\ov{n}}^2 \le CK \tr{\tilde{g}}{g}.
\end{split}
\end{equation}

From  the definition of $\Theta^{i\ov{i}}$, we have $\Theta^{n\ov{n}} \ge \frac{1}{n} \sum_i \Theta^{i\ov{i}} = \frac{1}{n} \tr{\tilde{g}}{g}$ and so
\begin{equation} \label{goodone}
\frac{1}{n} (\tr{\tilde{g}}{g}) u_{n\ov{n}}^2 \le C K^2 \tr{\tilde{g}}{g}.
\end{equation}
But the assumption $\lambda_2 \le (1-\delta) \lambda_n$ together with (\ref{defneta}), implies that
\[
\begin{split}
u_{n\ov{n}} \le {} & \sum_{i=1}^n \lambda_i - (n-1) \lambda_n+C \\
\le {} & \lambda_1+\lambda_2+(n-2)\lambda_n - (n-1) \lambda_n+C \\
\le {} & \lambda_1 - \delta \lambda_n+C
\le  - \frac{\delta}{2} \lambda_n ,
\end{split}
\]
where we use the following: since $\lambda_1 \le \cdots \le \lambda_n$ and the equation gives us $\lambda_1 \cdots \lambda_n = e^{\tilde{F}}$ (recall that $\tilde{F}=F+b$ is bounded), we may assume without loss of generality that $\lambda_1 << 1$, say (so in particular, $\tr{\tilde{g}}{g}\geq 1$). Hence we have $\delta^2 \lambda_n^2\leq 4u_{n\ov{n}}^2$
which from (\ref{goodone}) gives the uniform bound
$$\lambda_n \le C K.$$
By (\ref{ie}), this implies the estimate $\eta_{1\ov{1}} \le C'K$.
This completes the proof of Theorem \ref{theoremC2} in Case 1.

\bigskip
\noindent
{\bf Case 2:  $\lambda_2 \ge (1-\delta) \lambda_n$.}

\bigskip
From \eqref{defneta} we immediately get
\begin{equation} \label{apprbds}
\begin{split}
& | u_{i\ov{j}}| \le  C \quad \textrm{for } i \neq j, \\
& |u_{1\ov{1}}-\eta_{1\ov{1}}|\leq C.
\end{split}
\end{equation}

At this point we throw away two nonnegative terms in \eqref{secondcombine} to obtain
\begin{equation}\label{secondcombine3}
\begin{split}
0\geq {} &(1+2c)\sum_{i,j} \frac{ \tilde{g}^{i\ov{i}} \tilde{g}^{j\ov{j}} ( g_{j\ov{i}} \sum_a u_{a\ov{a} \ov{1}} - u_{j\ov{i} \ov{1}} + \hat{h}_{j\ov{i} \ov{1}} ) (g_{i\ov{j}} \sum_b u_{b\ov{b} 1} - u_{i\ov{j}1}  + \hat{h}_{i\ov{j}1} )}{{(n-1)^2 \eta_{1\ov{1}}}} \\
& {}  - (1+2c)\sum_{i} \frac{ \Theta^{i\ov{i}}|\eta_{1\ov{1} i}|^2}{(\eta_{1\ov{1}})^2}
+ \psi'' \sum_i \Theta^{i\ov{i}} |u_i|^2 \\
& {} + \varphi'' \sum_i \Theta^{i\ov{i}} \left| \sum_p u_p u_{\ov{p} i} + \sum_p u_{pi} u_{\ov{p}} \right|^2
 +  \frac{9}{10} \varphi' \sum_{i,p} \Theta^{i\ov{i}} ( |u_{pi}|^2 + |u_{p \ov{i}}|^2) \\
 & {} - C_0 \tr{\tilde{g}}{g}-C
   +\frac{2(1+2c)}{\eta_{1\ov{1}}} \sum_{i=2}^n  \Theta^{i\ov{i}} \textrm{Re} \left( \ov{T^1_{1i}} \eta_{1\ov{1}i} \right)
+ (C_1+1) \tr{\tilde{g}}{h},
\end{split}
\end{equation}

We wish to deal with the bad term $- (1+2c)\sum_{i} \frac{ \Theta^{i\ov{i}}|\eta_{1\ov{1} i}|^2}{(\eta_{1\ov{1}})^2}$ in \eqref{secondcombine3}. We first consider the summand when the index $i$ is equal to $1$.
Compute, using (\ref{Qi}),
\begin{equation} \label{Line21}
\begin{split}
\lefteqn{- (1+2c)^2 \frac{\Theta^{1\ov{1}} | \eta_{1\ov{1}1}|^2}{(\eta_{1\ov{1}})^2}} \\ = {} & - \Theta^{1\ov{1}} \left| \varphi' \left(\sum_p u_p u_{\ov{p}1} + \sum_p u_{p1} u_{\ov{p}} \right) + \psi' u_1 \right|^2 \\
\ge {} & - 2 \Theta^{1\ov{1}} (\varphi')^2 \left| \sum_p u_p u_{\ov{p}1} + \sum_p u_{p1} u_{\ov{p}} \right|^2 - 2 (\psi')^2 K \Theta^{1\ov{1}} \\
\ge  {} & - \varphi'' \Theta^{1\ov{1}} \left| \sum_p u_p u_{\ov{p}1} + \sum_p u_{p1} u_{\ov{p}} \right|^2 - C,
\end{split}
\end{equation}
where we have used that (\ref{hmwc1}) and (\ref{psidp}),  that $\Theta^{1\ov{1}}\approx (\eta_{1\ov{1}})^{-1}$.  Recall that we may assume without loss of generality that $\eta_{1\ov{1}} >> K$.

Next we exploit again the good first term on the third line of (\ref{secondcombine3}) together with the as-yet unused good second term on the same line to kill a ``small part'' of the bad term
$$-(1+2c) \sum_{i=2}^n \frac{\Theta^{i\ov{i}}|\eta_{1\ov{1}i}|^2}{(\eta_{1\ov{1}})^2}.$$
First, observe that applying again (\ref{Qi}) and the fact that $\varphi'' = 2(\varphi')^2$,
\begin{equation} \label{uno1}
\begin{split}
\lefteqn{\varphi'' \sum_{i =2}^n \Theta^{i\ov{i}} \left| \sum_p u_p u_{\ov{p} i} + \sum_p u_{pi} u_{\ov{p}} \right|^2 } \\= {} & 2 \sum_{i =2}^n \Theta^{i\ov{i}} \left| (1+2c)\frac{\eta_{1\ov{1} i}}{\eta_{1\ov{1}}} + \psi' u_i \right|^2 \\
\ge {} & 2(1+2c)^2 \delta \sum_{i = 2}^n \frac{\Theta^{i\ov{i}} |\eta_{1\ov{1}i}|^2}{(\eta_{1\ov{1}})^2} - \frac{2\delta (\psi')^2}{1-\delta} \sum_{i = 2}^n \Theta^{i\ov{i}} |u_i|^2,
\end{split}
\end{equation}
using \cite[Proposition 2.3]{HMW}.  Recall that  $\delta \leq \frac{1}{1+2A}$ and so from (\ref{psidp}),
\begin{equation} \label{dos}
 \frac{2\delta (\psi')^2}{1-\delta} \sum_{i = 2}^n \Theta^{i\ov{i}} |u_i|^2 \le \psi'' \sum_{i = 2}^n \Theta^{i\ov{i}} |u_i|^2 \le \psi'' \sum_{i=1}^n \Theta^{i\ov{i}} |u_i|^2.
 \end{equation}
Putting this together gives
\begin{equation} \label{uno}
\begin{split}
\lefteqn{\varphi'' \sum_{i =2}^n \Theta^{i\ov{i}} \left| \sum_p u_p u_{\ov{p} i} + \sum_p u_{pi} u_{\ov{p}} \right|^2 } \\ &  \ge 2(1+2c)^2 \delta \sum_{i = 2}^n \frac{\Theta^{i\ov{i}} |\eta_{1\ov{1}i}|^2}{(\eta_{1\ov{1}})^2} - \psi'' \sum_{i=1}^n \Theta^{i\ov{i}} |u_i|^2.
\end{split}
\end{equation}
Combining \eqref{Line21} and \eqref{uno} we get
\begin{equation} \label{due}
\begin{split}
 - (1+2c)\sum_{i} \frac{ \Theta^{i\ov{i}}|\eta_{1\ov{1} i}|^2}{(\eta_{1\ov{1}})^2}
+ \psi'' \sum_i \Theta^{i\ov{i}} |u_i|^2+ \varphi'' \sum_i \Theta^{i\ov{i}} \left| \sum_p u_p u_{\ov{p} i} + \sum_p u_{pi} u_{\ov{p}} \right|^2\\
\geq- (1+2c-2(1+2c)^2\delta)\sum_{i=2}^n \frac{ \Theta^{i\ov{i}}|\eta_{1\ov{1} i}|^2}{(\eta_{1\ov{1}})^2} -C.
\end{split}
\end{equation}
We will now make use of the good first line  of \eqref{secondcombine3}.  In fact we will only need a part of this term, namely the summands with $j=1$ and $i=2, \ldots, n$, which equal
$$\frac{1+2c}{(n-1)^2 \eta_{1\ov{1}}} \sum_{i=2}^n \tilde{g}^{i\ov{i}} \tilde{g}^{1\ov{1}} | u_{i\ov{1}1} - \hat{h}_{i\ov{1} 1}|^2.$$
We have from (\ref{comm}),
$$u_{i\ov{1}1} - \hat{h}_{i\ov{1}1}  = \eta_{1\ov{1}i} +\mathcal{E}_i$$
where
\begin{equation} \label{Eifor}
\mathcal{E}_i =  - \sum_p T^p_{1i} u_{p\ov{1}}  -\hat{h}_{i\ov{1}1} - (\tr{g}{h})_i + \hat{h}_{1\ov{1}i}.
\end{equation}
Hence
\begin{equation} \label{distribute}
\begin{split}
\lefteqn{\frac{1+2c}{(n-1)^2 \eta_{1\ov{1}}} \sum_{i=2}^n \tilde{g}^{i\ov{i}} \tilde{g}^{1\ov{1}} | u_{i\ov{1} 1} - \hat{h}_{i\ov{1} 1} |^2 } \\= {} & \frac{1+2c}{(n-1)^2 \eta_{1\ov{1}}}\sum_{i=2}^n  \tilde{g}^{i\ov{i}} \tilde{g}^{1\ov{1}} | \eta_{1\ov{1}i} + \mathcal{E}_i|^2 \\
={} &  \frac{1+2c}{(n-1)^2\eta_{1\ov{1}}}\sum_{i=2}^n  \tilde{g}^{i\ov{i}} \tilde{g}^{1\ov{1}} | \eta_{1\ov{1}i} |^2+  \frac{1+2c}{(n-1)^2\eta_{1\ov{1}}}\sum_{i=2}^n  \tilde{g}^{i\ov{i}} \tilde{g}^{1\ov{1}} |\mathcal{E}_i|^2\\
& +  \frac{2(1+2c)}{(n-1)^2\eta_{1\ov{1}}} \textrm{Re} \left\{ \sum_{i=2}^n  \tilde{g}^{i\ov{i}} \tilde{g}^{1\ov{1}} \eta_{1\ov{1}i} \ov{\mathcal{E}_i} \right\}.
\end{split}
\end{equation}
We claim that
\begin{equation} \label{niceclaim}
\frac{1+2c}{(n-1)^2 \eta_{1\ov{1}}} \sum_{i=2}^n \tilde{g}^{i\ov{i}} \tilde{g}^{1\ov{1}} |\eta_{1\ov{1}i}|^2 \ge {}  (1+2c- (1+2c)^2\delta) \sum_{i=2}^n \frac{\Theta^{i\ov{i}} |\eta_{1\ov{1}i}|^2}{(\eta_{1\ov{1}})^2}.
\end{equation}
Indeed this follows from the stronger inequality
\begin{equation} \label{stronger1}
\left(1- (1+2c)\delta\right) \frac{\Theta^{i\ov{i}}}{\eta_{1\ov{1}}} \le \frac{1}{(n-1)^2} \tilde{g}^{i\ov{i}} \tilde{g}^{1\ov{1}}, \quad \textrm{for } i=2, \ldots, n,
\end{equation}
which was proved in \cite[(4.41)]{TW4} with $3\delta/2$ instead of $(1+2c) \delta$.  The argument is exactly the same and so we will not repeat it here.

The next term that we need to bound is the cross term of (\ref{distribute}):
$$\frac{2(1+2c)}{(n-1)^2\eta_{1\ov{1}}} \textrm{Re} \left\{ \sum_{i=2}^n  \tilde{g}^{i\ov{i}} \tilde{g}^{1\ov{1}} \eta_{1\ov{1}i} \ov{\mathcal{E}_i} \right\}. $$
From (\ref{Eifor}), we have for $i=2, \ldots, n$,
\begin{equation} \label{almostlastclaim}
\mathcal{E}_i = \eta_{1\ov{1}} T_{i1\ov{1}}  + \tilde{\mathcal{E}}_i,
\end{equation}
where
$$\tilde{\mathcal{E}}_i = O(1),$$
thanks to \eqref{apprbds}.  Here we are writing $T_{ij\ov{\ell}}=g_{k\ov{\ell}} T^k_{ij}$.

Now using the facts that for $i=2, \ldots, n$, the quantity $\tilde{g}^{i\ov{i}}$ is comparable to $(\eta_{1\ov{1}})^{-1}$ and $\tilde{g}^{1\ov{1}}$ is comparable to $\Theta^{i\ov{i}}$, we have
\[\begin{split}
\lefteqn{\frac{2(1+2c)}{(n-1)^2 \eta_{1\ov{1}}} \textrm{Re} \left\{ \sum_{i=2}^n \tilde{g}^{i\ov{i}} \tilde{g}^{1\ov{1}} \eta_{1\ov{1} i} \ov{\tilde{\mathcal{E}}_i} \right\}}\\
\ge {} & - \frac{C}{(\eta_{1\ov{1}})^2} \sum_{i=2}^n \tilde{g}^{1\ov{1}} |\eta_{1\ov{1}i}|  \\
\ge {} & -  \frac{\delta}{4} \sum_{i=2}^n \Theta^{i\ov{i}} \frac{| \eta_{1\ov{1}i}|^2}{(\eta_{1\ov{1}})^2} - \frac{C'}{\delta (\eta_{1\ov{1}})^2} \tr{\tilde{g}}{g}.
\end{split}\]
Without loss of generality, we may assume that $(\eta_{1\ov{1}})^2\geq 2C'/\delta$, so that
\begin{equation}\label{Line13}
\frac{2(1+2c)}{(n-1)^2 \eta_{1\ov{1}}} \textrm{Re} \left\{ \sum_{i=2}^n \tilde{g}^{i\ov{i}} \tilde{g}^{1\ov{1}} \eta_{1\ov{1} i} \ov{\tilde{\mathcal{E}}_i} \right\}\geq
-  \frac{\delta}{4} \sum_{i=2}^n \Theta^{i\ov{i}} \frac{| \eta_{1\ov{1}i}|^2}{(\eta_{1\ov{1}})^2} - \frac{1}{2} \tr{\tilde{g}}{g}.
\end{equation}
Next we make the observation that, assuming without loss of generality $\lambda_n>>\lambda_1$, we have
 for $i=2, \ldots, n$,
\begin{equation}\label{obs}
\left| \tilde{g}^{i\ov{i}} - \frac{(n-1)}{\eta_{1\ov{1}}} \right| \le \frac{2(n-1)\delta}{\eta_{1\ov{1}}}, \quad \textrm{for $i=2, \ldots, n$.}
\end{equation}

Using \eqref{torsionbound}, \eqref{obs}, and the fact that $\ti{g}^{1\ov{1}}\leq (n-1)\Theta^{i\ov{i}}$ for $i\geq 2$, we obtain
\begin{equation} \label{Line14}
\begin{split}
\lefteqn{\frac{2(1+2c)}{(n-1)^2 \eta_{1\ov{1}}} \textrm{Re} \left\{ \sum_{i=2}^n \tilde{g}^{i\ov{i}} \tilde{g}^{1\ov{1}} \eta_{1\ov{1} i} \eta_{1\ov{1}} \ov{T_{i1\ov{1}}} \right\} }\\
\ge {} & \frac{2(1+2c)\tilde{g}^{1\ov{1}}}{(n-1) \eta_{1\ov{1}}} \textrm{Re} \left\{ \sum_{i=2}^n  \eta_{1\ov{1} i}  \ov{T_{i1\ov{1}}}  \right\} \\
&  - \frac{4(1+2c)\sqrt{C_2}\delta}{\eta_{1\ov{1}}}  \sum_{i=2}^n \Theta^{i\ov{i}} |\eta_{1\ov{1}i}|   \\
\ge {} & \frac{2(1+2c)\tilde{g}^{1\ov{1}}}{(n-1) \eta_{1\ov{1}}} \textrm{Re} \left\{ \sum_{i=2}^n  \eta_{1\ov{1} i}   \ov{T_{i1\ov{1}}}   \right\} \\
&  - \frac{\delta}{4} \sum_{i=2}^n \Theta^{i\ov{i}} \frac{|\eta_{1\ov{1}i}|^2}{(\eta_{1\ov{1}})^2}  - 16(1+2c)^2C_2\delta \tr{\tilde{g}}{g}.
\end{split}
\end{equation}
Putting together (\ref{distribute}), (\ref{niceclaim}), (\ref{almostlastclaim}), (\ref{Line13}), (\ref{Line14}) we obtain the lower bound
\begin{equation}\label{tre}
\begin{split}
\lefteqn{\frac{1+2c}{(n-1)^2 \eta_{1\ov{1}}} \sum_{i=2}^n \tilde{g}^{i\ov{i}} \tilde{g}^{1\ov{1}} | u_{i\ov{1}1} - \hat{h}_{i\ov{1} 1}|^2 } \\\ge {} & (1+2c- (1+2c)^2 \delta - \frac{\delta}{2}) \sum_{i=2}^n \frac{\Theta^{i\ov{i}} |\eta_{1\ov{1}i}|^2}{(\eta_{1\ov{1}})^2}
 \\ & +
\frac{2(1+2c)\tilde{g}^{1\ov{1}}}{(n-1) \eta_{1\ov{1}}} \textrm{Re} \left\{ \sum_{i=2}^n  \eta_{1\ov{1} i}   \ov{T_{i1\ov{1}}}   \right\}  
  - 16(1+2c)^2C_2\delta \tr{\tilde{g}}{g} -\frac{1}{2} \tr{\tilde{g}}{g}.
\end{split}
\end{equation}
Combining \eqref{secondcombine3}, \eqref{due} and \eqref{tre} we conclude that
\begin{equation}\label{quattro}
\begin{split}
0\geq {} &\frac{\delta}{2}\sum_{i=2}^n \Theta^{i\ov{i}} \frac{|\eta_{1\ov{1}i}|^2}{(\eta_{1\ov{1}})^2} -\frac{1}{2} \tr{\tilde{g}}{g} - 16(1+2c)^2C_2\delta \tr{\tilde{g}}{g} \\
& {}    + (C_1+1) \tr{\tilde{g}}{h}
 +  \frac{9}{10} \varphi' \sum_{i,p} \Theta^{i\ov{i}} ( |u_{pi}|^2 + |u_{p \ov{i}}|^2)- C_0 \tr{\tilde{g}}{g}-C\\
&+\frac{2(1+2c)\tilde{g}^{1\ov{1}}}{(n-1) \eta_{1\ov{1}}} \textrm{Re} \left\{ \sum_{i=2}^n  \eta_{1\ov{1} i}  \ov{T_{i1\ov{1}}}  \right\}
 +\frac{2(1+2c)}{\eta_{1\ov{1}}} \sum_{i=2}^n  \Theta^{i\ov{i}} \textrm{Re} \left( \ov{T^1_{1i}} \eta_{1\ov{1}i} \right).
\end{split}
\end{equation}
We now deal with the last two terms in this equation.
Recall that
$$\Theta^{i\ov{i}} = \frac{1}{n-1} \tilde{g}^{1\ov{1}} + O((\eta_{1\ov{1}})^{-1}), \quad \textrm{for } i=2, \ldots, n,$$
and in particular, $\Theta^{i\ov{i}}$ is large when $i\neq 1$.
Then the last term in \eqref{quattro} satisfies
\begin{equation} \label{line531}
\begin{split}
\lefteqn{ \frac{2(1+2c)}{\eta_{1\ov{1}}} \sum_{i=2}^n  \Theta^{i\ov{i}} \textrm{Re} \left( \ov{T^1_{1i}} \eta_{1\ov{1}i} \right)  } \\ \geq {} &
 \frac{2(1+2c)}{(n-1)\eta_{1\ov{1}}} \sum_{i=2}^n \tilde{g}^{1\ov{1}} \textrm{Re} \left( \eta_{1\ov{1}i} \ov{T_{1i\ov{1}}} \right) - \frac{C}{(\eta_{1\ov{1}})^2} \sum_{i=2}^n | \eta_{1\ov{1}i}| \\
\ge {} &- \frac{2(1+2c)}{(n-1)\eta_{1\ov{1}}} \sum_{i=2}^n \tilde{g}^{1\ov{1}} \textrm{Re} \left( \eta_{1\ov{1}i} \ov{T_{i1\ov{1}}} \right)
- \frac{\delta}{2} \sum_{i=2}^n \Theta^{i\ov{i}} \frac{| \eta_{1\ov{1}i}|^2}{(\eta_{1\ov{1}})^2} -\frac{1}{\delta (\eta_{1\ov{1}})^2},
\end{split}
\end{equation}
and the first term on the right hand side precisely cancels the other torsion term in \eqref{quattro}. We can also assume that $(\eta_{1\ov{1}})^2\geq\frac{1}{\delta}$.
Combining \eqref{quattro} and \eqref{line531}, we get
\begin{equation}
\begin{split}
0\geq {} & -\frac{1}{2} \tr{\tilde{g}}{g} - 16(1+2c)^2C_2\delta \tr{\tilde{g}}{g} + (C_1+1) \tr{\tilde{g}}{h}\\
& + \frac{9}{10} \varphi' \sum_{i,p} \Theta^{i\ov{i}} ( |u_{pi}|^2 + |u_{p \ov{i}}|^2)- C_0 \tr{\tilde{g}}{g}-C.
\end{split}
\end{equation}
But recall that $(C_1 +1) \tr{\tilde{g}}{h} \ge (C_0+2)\tr{\tilde{g}}{g}$ and $\varphi'>0$ so 
\begin{equation}
\begin{split}
0\geq {} & - 16(1+2c)^2C_2\delta \tr{\tilde{g}}{g} +  \frac{3}{2}\tr{\tilde{g}}{g}
-C.
\end{split}
\end{equation}
We chose $\delta$ so that $16(1+2c)^2C_2 \delta \le 1$, and so we conclude that $\tr{\tilde{g}}{g}$ is bounded from above at the maximum, and in particular,
$$\frac{1}{\lambda_1}\leq C.$$
From the assumption $\lambda_j\geq (1-\delta)\lambda_n$ for $j\geq 2$, together with the Monge-Amp\`ere equation, we conclude that
$$\lambda_n^{n-1}\leq C\lambda_2\dots\lambda_n\leq C\frac{\lambda_2\dots\lambda_n}{\lambda_1\dots\lambda_n}=\frac{C}{\lambda_1}\leq C.$$
Since $n\geq 3$, we conclude that $\lambda_n\leq C$ and so $\eta_{1\ov{1}}\leq C$ and we are done.

\section{Proofs of Theorem \ref{mainthm} and its Corollaries} \label{sectionproofs}

\begin{proof}[Proof of Theorem \ref{mainthm}]
Assume that we are in the setting of Theorem \ref{mainthm}.
We claim that we have an {\em a priori} gradient estimate
\begin{equation}\label{grad}
\sup_M |\nabla u|_g\leq C,
\end{equation}
where the constant $C$ depends only on
$\|F\|_{C^2(M,g)}$ and the fixed data $(M,\omega_0),$ $\omega$.
Indeed, thanks to the estimates from Theorems \ref{linfty} and \ref{theoremC2}, the proof of \eqref{grad} is identical to \cite[Theorem 5.1]{TW4}, since the fact that $\omega$ was assumed to be K\"ahler there played no role in the proof.   The basic ideas are as follows.  The estimate
$$\tr{g}{\tilde{g}} \le C (\sup_M | \nabla u|^2_g +1),$$
of
 Theorem \ref{theoremC2}  is compatible with a blow-up argument (cf. \cite{HMW}) so we can apply the Liouville-type theorem \cite[Theorem 5.2]{TW4} that a $(n-1)$-PSH function on $\mathbb{C}^n$ which is Lipschitz continous and maximal with bounded $L^{\infty}$ and Lipschitz norms must be constant.  The proof of the Liouville Theorem uses key  ideas of Dinew-Ko{\l}odziej \cite{DK}.

Given (\ref{grad}), we follow \cite[Theorem 6.1]{TW4} (see also \cite{FWW1,FWW2}) to derive higher order estimates
$$\| u \|_{C^k(M,g)} \le C_k,$$
for $k=0,1,2, \ldots$, and
 $$\tilde{g} \ge \frac{1}{C_0} g,$$
where each $C_k$ is a positive constant which depends only on $k$ and the fixed data $(M,\omega_0), \omega$ and $F$.

Indeed, combining the results of Theorems \ref{linfty} and \ref{theoremC2} with \eqref{grad}, we have the following estimates:
\begin{equation}\label{bds}
\sup_M |u| + \sup_M | \partial u|_g + \sup_M | \partial \ov{\partial} u |_g \le C,
\end{equation}
and from the equation (\ref{ma1}) and \eqref{boundb}, the uniform upper bound on $\tr{g}{\tilde{g}}$ gives
\begin{equation}\label{bds2}
C^{-1} g \le \tilde{g} \le C g.
\end{equation}
By the standard linear elliptic theory, it suffices to obtain a $C^{2+\alpha}(M,g)$ bound for $u$ for some $\alpha>0$.
This can be done with the usual Evans-Krylov method, adapted to the complex setting (see \cite{Si, Tr}), and the details are the same
as in \cite[Theorem 6.1]{TW4} (see also Section \ref{sectionreduce} below).

To finish the proof of Theorem \ref{mainthm}, it remains to  set up a continuity method and establish ``openness'', and prove the uniqueness of the solution. The proofs of these items are identical to the ones given in \cite[Section 6]{TW4}, where the assumption that $\omega$ was K\"ahler was never used.
\end{proof}

\begin{proof}[Proof of Corollary \ref{cor1}]
As we remarked in the introduction, if $\omega$ is Astheno-K\"ahler, $\omega_0$ is Gauduchon, and $u$ is a smooth function such that
$$\Psi_u=\omega_0^{n-1}+\ddbar u\wedge\omega^{n-2}>0,$$
then $\de\db\Psi_u=0$ too. A simple linear algebra argument (see \cite{Mi}) shows that there exists a unique Hermitian metric $\omega_u$ such that
$\omega_u^{n-1}=\Psi_u$. It follows that $\omega_u$ is Gauduchon. Similarly, if $\omega_0$ is strongly Gauduchon then so is $\omega_u$.
Then it is easy to see (cf. \cite{FWW2} or \cite[Section 2]{TW4}) that \eqref{mma1} is equivalent to \eqref{ma1} if we set $F'=F/(n-1), b'=b/(n-1)$. Theorem \ref{mainthm} therefore implies Corollary \ref{cor1}.
\end{proof}

\begin{proof}[Proof of Corollary \ref{cor2}]
By assumption, there is a smooth function $F$ such that
$$\Ric(\omega)=\psi+\ddbar F.$$
We use Corollary \ref{cor1} to solve the Calabi-Yau equation
$$\ti{\omega}^n=e^{F+b}\omega^n,$$
for some constant $b$, with $\ti{\omega}$ a Gauduchon metric.
Taking $-\ddbar\log$ of this equation gives
$$\Ric(\ti{\omega})=\Ric(\omega)-\ddbar F=\psi,$$
as required.
\end{proof}

\begin{proof}[Proof that Conjecture \ref{conjtw} implies Conjecture \ref{conjG}]  For $n=2$ this follows from the result of Cherrier \cite{Ch}. For $n>2$, this is a combination of the arguments of the previous two proofs. Fix a Gauduchon metric $\omega$ on $M$.
By assumption we have $$\textrm{Ric}(\omega) = \psi + \frac{1}{n-1}\ddbar F,$$ for a smooth function $F$.  With this choice of $F$, let $\omega_0=\omega$ and let $(u,b)$ solve (\ref{manew}) with $\Phi_u >0$.  Then for $\tilde{\omega}$ the $(n-1)$th root of $\Phi_u$, we have
$$\tilde{\omega}^n = e^{(F+b)/(n-1)} \omega^n,$$
and hence 
$\textrm{Ric}(\tilde{\omega}) =\psi.$
Since $\Phi_u$ is $\partial \ov{\partial}$-closed, $\tilde{\omega}$ is Gauduchon.
\end{proof}

\section{An $L^{\infty}$ estimate for (\ref{manew}).} \label{sectionnew}

In this section we prove Theorem \ref{Linftythm}.  The arguments use the same basic ideas as in Section \ref{sectionzero} except that we have to deal with the extra term $\mathrm{Re}\left(\mn\de u\wedge \db(\omega^{n-2})\right)$ in $\Phi_u$, and we make use of a crucial cancellation between certain torsion terms.  Recall that
$$\Phi_u = \omega_0^{n-1}+\ddbar u\wedge\omega^{n-2}+\mathrm{Re}\left(\mn\de u\wedge \db(\omega^{n-2})\right)>0,$$
with $u$  normalized by $\sup_M u=0$.
Set
$$E=\frac{1}{(n-1)!}\mathrm{Re}\left(\mn\de u\wedge \db(\omega^{n-2})\right).$$
Define $\ti{\omega}:=\frac{1}{(n-1)!}*\Phi_u$, for $*$ the Hodge star operator of $\omega$.  This is a Hermitian metric, and can be written
\begin{equation}\label{metrn}
\ti{\omega}=\omega_h+\frac{1}{n-1}\left((\Delta u)\omega-\ddbar u\right)+*E,
\end{equation}
where $\omega_h$ is the Hermitian metric defined by
$$\omega_h=\frac{1}{(n-1)!}*\omega_0^{n-1}.$$
Let us also define $H=\tr{\omega}{(*E)}$, so that
\begin{equation} \label{stare}
n *E\wedge \omega^{n-1}=n! E\wedge\omega=H\omega^n.
\end{equation}
Then taking the trace of \eqref{metrn} we see that
\begin{equation}\label{lapln}
\tr{\omega}{\ti{\omega}}=\tr{\omega}{\omega_h}+\Delta u +H,
\end{equation}
and therefore
\begin{equation}\label{ddbn}
\ddbar u =(n-1)\omega_h+ (\tr{\omega}{\ti{\omega}}-\tr{\omega}{\omega_h}-H)\omega-(n-1)\ti{\omega}+(n-1)*E.
\end{equation}
Since $\Phi_u$ satisfies $\det \Phi_u = e^{F+b} \det (\omega^{n-1})$ it follows that
 $\ti{\omega}$ satisfies the Monge-Amp\`ere equation
\begin{equation}\label{ma}
\ti{\omega}^n=e^{F+b}\omega^n.
\end{equation}

As in Section \ref{sectionzero}, we have $|b|\leq\sup_M|F|+C$,
for a uniform constant $C$ which depends only on $(M,\omega_0)$ and $\omega$.
To prove the $L^{\infty}$ estimate of $u$, we first prove the following lemma, which is the analogue of (\ref{need}).

\begin{lemma}  For a uniform $C$, we have
\begin{equation}\label{needn}
\begin{split}
\lefteqn{\ddbar u \wedge (2\omega_0^{n-1}+\ddbar u\wedge\omega^{n-2}+(n-1)!E)} \\ & \leq C(1+|\nabla u|^2)\omega^n-(n-1)!\ddbar u\wedge E.
\end{split}
\end{equation}
\end{lemma}
\begin{proof}
Observe that
\begin{equation}\label{crude}
|H|\leq C|\nabla u|,\quad |E|\leq C|\nabla u|.
\end{equation}
Let us define
\[
\begin{split}
\lefteqn{S=((n-1)\omega_h+ (\tr{\omega}{\ti{\omega}}-\tr{\omega}{\omega_h})\omega-(n-1)\ti{\omega})} \\ & \wedge (2\omega_0^{n-1}+((n-1)\omega_h+ (\tr{\omega}{\ti{\omega}}-\tr{\omega}{\omega_h})\omega-(n-1)\ti{\omega})\wedge\omega^{n-2}).
\end{split}
\]
Then using  Lemma \ref{lemmatw}, we have
\begin{equation}\label{estn}
S\leq C\omega^n.
\end{equation}
We now start with a direct calculation, using \eqref{stare}, \eqref{ddbn}, \eqref{crude} and \eqref{estn},
\begin{equation}\label{calc1}
\begin{split}
&\ddbar u \wedge (2\omega_0^{n-1}+\ddbar u\wedge\omega^{n-2}+(n-1)!E)\\
&=S+(n-1)!\ddbar u\wedge E-(n-1)H\omega_h\wedge\omega^{n-1}+(n-1)^2*E\wedge\omega_h\wedge\omega^{n-2}\\
&-(\tr{\omega}{\ti{\omega}})H\omega^n+(n-1)(\tr{\omega}{\ti{\omega}})*E\wedge\omega^{n-1}
+(\tr{\omega}{\omega_h})H\omega^n-(n-1)(\tr{\omega}{\omega_h})*E\wedge\omega^{n-1}\\
&-2H\omega_0^{n-1}\wedge\omega-(n-1)H\omega_h\wedge\omega^{n-1}-H(\tr{\omega}{\ti{\omega}})\omega^n
+H(\tr{\omega}{\omega_h})\omega^n+H^2\omega^n\\
&+(n-1)H\ti{\omega}\wedge\omega^{n-1}-(n-1)H*E\wedge\omega^{n-1}+(n-1)H\ti{\omega}\wedge\omega^{n-1}\\
&-(n-1)^2*E\wedge\ti{\omega}\wedge\omega^{n-2}+2(n-1)*E\wedge\omega_0^{n-1}+(n-1)^2*E\wedge\omega_h\wedge\omega^{n-2}\\
&+(n-1)(\tr{\omega}{\ti{\omega}})*E\wedge\omega^{n-1}-(n-1)(\tr{\omega}{\omega_h})*E\wedge\omega^{n-1}-(n-1)H*E\wedge\omega^{n-1}\\
&-(n-1)^2 *E\wedge\ti{\omega}\wedge\omega^{n-2}+(n-1)^2*E\wedge*E\wedge\omega^{n-2}\\
&\leq (n-1)!\ddbar u\wedge E+C(1+|\nabla u|^2)\omega^n\\
&+\frac{2n-4}{n}(\tr{\omega}{\ti{\omega}})H\omega^n -2(n-1)^2*E\wedge\ti{\omega}\wedge\omega^{n-2}.
\end{split}
\end{equation}
A simple calculation shows that
\begin{equation}\label{star}
*(\ti{\omega}\wedge\omega^{n-2})=(n-2)!\big((\tr{\omega}{\ti{\omega}})\omega-\ti{\omega}\big).
\end{equation}
Indeed, we can compute at a point where $\omega=\sum_i e_i$ and $\ti{\omega}=\sum_j \lambda_j e_j$, where $e_i=\mn dz^i\wedge d\ov{z}^i$, and then
\[\begin{split}
\ti{\omega}\wedge\omega^{n-2}&=(n-2)!\left(\sum_j \lambda_j e_j\right)\wedge\sum_{k<\ell} e_1\dots \widehat{e}_k\dots \widehat{e}_\ell\dots e_n\\
&=(n-2)!\sum_k\left(\sum_{j\neq k} \lambda_j \right)e_1\dots \widehat{e}_k\dots e_n,
\end{split}\]
from which \eqref{star} follows. Therefore
\[\begin{split}
-2(n-1)^2*E\wedge\ti{\omega}\wedge\omega^{n-2}&= -2(n-1)^2E\wedge*(\ti{\omega}\wedge\omega^{n-2})\\
&=-2(n-1)(n-1)! (\tr{\omega}{\ti{\omega}}) E\wedge\omega+2(n-1)(n-1)!E\wedge\ti{\omega}\\
&=-\frac{2n-2}{n} (\tr{\omega}{\ti{\omega}})H\omega^n+2(n-1)(n-1)!E\wedge\ti{\omega},
\end{split}\]
and, using again \eqref{stare},
\begin{equation}\label{calc2}
\begin{split}
&\frac{2n-4}{n}(\tr{\omega}{\ti{\omega}})H\omega^n -2(n-1)^2*E\wedge\ti{\omega}\wedge\omega^{n-2}\\
={} &-\frac{2}{n}(\tr{\omega}{\ti{\omega}})H\omega^n+2(n-1)(n-1)!E\wedge\ti{\omega}\\
= {} & -\frac{2}{n}(\tr{\omega}{\omega_h})H\omega^n-\frac{2}{n}H\Delta u\omega^n -\frac{2}{n}H^2\omega^n+
2(n-1)(n-1)!E\wedge\omega_h \\ & + 2(n-1)!\Delta u E\wedge\omega
-2(n-1)!E\wedge\ddbar u+2(n-1)(n-1)!E\wedge *E\\
= {} &-\frac{2}{n}(\tr{\omega}{\omega_h})H\omega^n-\frac{2}{n}H^2\omega^n+
2(n-1)(n-1)!E\wedge\omega_h\\
&-2(n-1)!E\wedge\ddbar u+2(n-1)(n-1)!E\wedge *E\\
\leq {} & -2(n-1)!E\wedge\ddbar u+C(1+|\nabla u|^2)\omega^n.
\end{split}
\end{equation}
Note that the two bad terms $-\frac{2}{n}H\Delta u\omega^n$ and $2(n-1)!\Delta u E\wedge\omega$ exactly canceled out.
Combining \eqref{calc1} and \eqref{calc2} concludes the proof of \eqref{needn}.
\end{proof}

\begin{proof}[Proof of Theorem \ref{Linftythm}]
Following the method of  Section \ref{sectionzero}, we
 use  \eqref{needn} in a Moser iteration argument. For $p$ sufficiently large, as in (\ref{c1m}),
\begin{equation}
\begin{split}
\lefteqn{  \int_M e^{-pu} \ddbar u \wedge \left( 2\omega_0^{n-1}+\ddbar u\wedge\omega^{n-2}+(n-1)!E   \right) } \\
\geq {} & \frac{1}{Cp} \int_M \sqrt{-1} \partial e^{-\frac{pu}{2}} \wedge \ov{\partial} e^{-\frac{pu}{2}} \wedge  \omega^{n-1}-\frac{C}{p}\int_M e^{-pu}\omega^n,
\end{split}
\end{equation}
where we have used
$$\omega_0^{n-1}+\ddbar u\wedge\omega^{n-2}+(n-1)!E>0,$$
and
$$\de\db\left(\ddbar u\wedge\omega^{n-2}+(n-1)!E\right)=0.$$
Using claim \eqref{needn}, we see that
\begin{equation}
\begin{split}
\lefteqn{\int_M e^{-pu} \ddbar u \wedge \left( 2\omega_0^{n-1}+\ddbar u\wedge\omega^{n-2}+(n-1)!E   \right) }  \\
 \le {} & C \int_M e^{-pu} \omega^n + C \int_M  e^{-pu}| \nabla u|^2 \omega^n -\int_M e^{-pu}\ddbar u \wedge \mathrm{Re}\left(\mn\de u\wedge \db(\omega^{n-2})\right).
\end{split}
\end{equation}
We have that
$$C\int_M  e^{-pu}| \nabla u|^2 \omega^n\leq\frac{C}{p^2}\int_M \sqrt{-1} \partial e^{-\frac{pu}{2}} \wedge \ov{\partial} e^{-\frac{pu}{2}} \wedge  \omega^{n-1}, $$
and
\begin{equation}
\begin{split}
-&\int_M e^{-pu}\ddbar u \wedge \mathrm{Re}\left(\mn\de u\wedge \db(\omega^{n-2})\right)\\
&=   -\mathrm{Re}\int_M e^{-pu}\ddbar u \wedge \mn\de u\wedge \db(\omega^{n-2})  \\
& =\frac{1}{p}\mathrm{Re}\int_M \mn\de(e^{-pu})\wedge\ddbar u \wedge \db (\omega^{n-2})\\
&=\frac{1}{p}\mathrm{Re}\int_M \mn\de(e^{-pu})\wedge\db u \wedge \ddbar (\omega^{n-2})\\
&=-\frac{4}{p^2}\int_M \sqrt{-1} \partial e^{-\frac{pu}{2}} \wedge \ov{\partial} e^{-\frac{pu}{2}} \wedge  \ddbar (\omega^{n-2})\\
&\leq\frac{C}{p^2}\int_M \sqrt{-1} \partial e^{-\frac{pu}{2}} \wedge \ov{\partial} e^{-\frac{pu}{2}} \wedge  \omega^{n-1}.
\end{split}
\end{equation}
Putting all these together, we conclude that for $p$ sufficiently large we have
$$\int_M |\de e^{-\frac{pu}{2}}|^2\omega^n\leq Cp\int_M e^{-pu}\omega^n.$$
From here we can conclude the proof exactly as in Theorem \ref{linfty}. Again, we have two different ways to complete the proof.
We can modify the first proof there by replacing the operator $\Delta'$ with
the operator $Lu=\Delta u +H(u)$, where as before
$$H(u)=\frac{n\mathrm{Re}\left(\mn\de u\wedge \db(\omega^{n-2})\wedge\omega\right)}{\omega^n}=\frac{n(n-2)}{(n-1)}\cdot\frac{\mathrm{Re}\left(\mn\de u\wedge \db(\omega^{n-1})\right)}{\omega^n}.$$
The assumption that $\omega$ is Gauduchon implies that $\int_M Lu\omega^n=0$, and we have
$$Lu\geq -\tr{\omega}{\omega_h}\geq -C.$$ Therefore one can use the Green's formula for $L$ in exactly the same way as before.

We can also modify the second proof of Theorem \ref{linfty}, using Moser iteration. The function $v=u-\inf_M u$ satisfies $\Delta v \geq -C -H(v)$.
Then, as in \eqref{comp}, for any $p\geq 1$ we compute
\begin{equation}
\begin{split}
\int_M |\de v^{\frac{p+1}{2}}|^2_g\omega^n&=\frac{(p+1)^2}{4p}\int_M  v^p(-\Delta v)\omega^n+\frac{n(p+1)}{4p}\int_M \mn\db v^{p+1}\wedge\de(\omega^{n-1})\\
&=\frac{(p+1)^2}{4p}\int_M  v^p(-\Delta v)\omega^n\\
&\leq C\frac{(p+1)^2}{4p}\int_M v^p\omega^n +\frac{n(n-2)(p+1)^2}{4p(n-1)}\int_M v^p\mn\de v\wedge \db(\omega^{n-1})\\
&= C\frac{(p+1)^2}{4p}\int_M v^p\omega^n +\frac{n(n-2)(p+1)}{4p(n-1)}\int_M \mn\de v^{p+1}\wedge \db(\omega^{n-1})\\
&= C\frac{(p+1)^2}{4p}\int_M v^p\omega^n,
\end{split}
\end{equation}
using twice the Gauduchon condition $\de\db(\omega^{n-1})=0$. From here we conclude exactly as in Theorem \ref{linfty}.
\end{proof}

\section{Proof of Theorem \ref{thmreduce}} \label{sectionreduce}

Finally, we give the proof of Theorem \ref{thmreduce}. Suppose that $u$ solves \eqref{manew}. Thanks to Theorem \ref{Linftythm}, we have $\|u\|_{L^\infty}\leq C$ for a uniform constant $C$ (which depends only on $\omega,\omega_0$ and $F$). Since we are assuming that \eqref{hmwtype} holds, we can use a blow-up argument as in \cite{TW4} (which in turn uses the ideas of Dinew-Ko{\l}odziej \cite{DK}) to show that
\begin{equation}\label{grad2}
\sup_M |\nabla u|_g\leq C,
\end{equation}
for a uniform constant $C$.

Indeed the blow-up argument of \cite[Theorem 5.1]{TW4} can be applied with minor modifications to give (\ref{grad2}). The only difference here is the presence  of the term $*E(u)$.  But this term is linear in $\partial u$ and hence converges to zero uniformly on compact subsets under the the rescaling procedure of \cite{TW4}.  The rest of the argument is identical to \cite{TW4}.

Combining $\|u\|_{L^\infty}\leq C$  with \eqref{grad2} and \eqref{hmwtype} we have
\begin{equation}\label{bds3}
\sup_M |u| + \sup_M | \partial u|_g + \sup_M | \partial \ov{\partial} u |_g \le C,
\end{equation}
and (\ref{manew}) together with the bound on $|b|$ and the uniform upper bound on $\tr{g}{\tilde{g}}$ give
\begin{equation}\label{bds4}
C^{-1} g \le \tilde{g} \le C g.
\end{equation}
By the standard linear elliptic theory, it suffices to obtain a $C^{2+\alpha}(M,g)$ bound for $u$ for some $\alpha>0$.
To do this, we again follow the strategy of \cite{TW4} using the Evans-Krylov theory (see also \cite{FWW2}). However, there are new difficulties that arise.
Let us define a tensor
$$Z=*E=\mn Z_{i\ov{j}} dz^i\wedge d\ov{z}^j,$$
so that
$$\tilde{g}_{i\ov{j}} = g_{i\ov{j}} + \frac{1}{n-1} \left( (\Delta u)g_{i\ov{j}} -u_{i\ov{j}} \right) +  Z_{i\ov{j}}.$$
We have $|Z|_g\leq C|\nabla u|_g$, and $|\nabla Z|_g\leq C(|\nabla\nabla u|_g+|\nabla\ov{\nabla}u|_g)$. We will also use the notation
$\hat{Z}=(n-1)Z$.

As in \cite{TW4}, we will work in a small open subset of $\mathbb{C}^n$, containing a ball $B_{2R}$ of radius $2R$.  
The equation is given by
$$\log \frac{\det \tilde{g}}{\det g} = \tilde{F},$$
where $\ti{F}=F+b$.
Let $\gamma=(\gamma^i)$ be a unit vector in $\mathbb{C}^n$.  The same calculation as in \eqref{atx}, using \eqref{comm}, gives
\begin{equation} \label{ek1}
\begin{split}
\Theta^{i\ov{j}} u_{\gamma \ov{\gamma}i\ov{j}} &\ge  G-C \sum_{p,q} | u_{\gamma p\ov{q}}|+S - C|u_{ij}|_g,
\end{split}\end{equation}
with
$$S=\frac{1}{(n-1)^2} \tilde{g}^{i\ov{q}} \tilde{g}^{p\ov{j}} ( g_{p\ov{q}} g^{r\ov{s}} u_{r\ov{s} \ov{\gamma}} - u_{p\ov{q} \ov{\gamma}} + \hat{B}_{p\ov{q}\ov{\gamma}} ) (g_{i\ov{j}} g^{a\ov{b}} u_{a\ov{b} \gamma} - u_{i\ov{j}\gamma} + \hat{B}_{i\ov{j} \gamma} ),$$
and
$$B_{i\ov{j} \gamma} = \hat{h}_{i\ov{j} \gamma} + \hat{Z}_{i\ov{j} \gamma},$$
and for $G$ a uniformly bounded function (which depends on $\gamma$), which may change from line to line.  Here we are writing  $|u_{ij}|^2_g$ for $| \nabla \nabla u|^2_g$.
We convert these covariant derivatives into partial derivatives and obtain
$$\Theta^{i\ov{j}} \partial_i \partial_{\ov{j}} u_{\gamma \ov{\gamma}} \ge G - C \sum_{p,q} | u_{p\ov{q}\gamma}| + S- C|u_{ij}|_g.$$
For a uniform $C'>0$,
\[\begin{split}
S&\geq C'^{-1} g^{i\ov{q}} g^{p\ov{j}} ( g_{p\ov{q}} g^{r\ov{s}} u_{r\ov{s} \ov{\gamma}} - u_{p\ov{q} \ov{\gamma}} + \hat{B}_{p\ov{q}\ov{\gamma}} ) (g_{i\ov{j}} g^{a\ov{b}} u_{a\ov{b} \gamma} - u_{i\ov{j}\gamma} + \hat{B}_{i\ov{j} \gamma})\\
&\geq C'^{-1}g^{i\ov{q}} g^{p\ov{j}}( g_{p\ov{q}} g^{r\ov{s}} u_{r\ov{s} \ov{\gamma}} - u_{p\ov{q} \ov{\gamma}})(g_{i\ov{j}} g^{a\ov{b}} u_{a\ov{b} \gamma} - u_{i\ov{j}\gamma})-C'| B_{i\ov{j} \gamma}|^2_g\\
&=C'^{-1}\left((n-2)|g^{a\ov{b}} u_{a\ov{b} \gamma}|^2+ |u_{i\ov{j} \gamma}|_g^2\right)-C' | B_{i\ov{j} \gamma}|^2_g\\
&\geq C \sum_{p,q} | u_{p\ov{q}\gamma}|-C' | B_{i\ov{j} \gamma}|^2_g-C',
\end{split}\]
and
$$|B_{i\ov{j} \gamma}|^2_g \le C(1 + |  u_{ij}|^2_g).$$
Hence, we conclude that
\begin{equation} \label{ek3}
\begin{split}
\Theta^{i\ov{j}} \de_i\de_{\ov{j}}u_{\gamma \ov{\gamma}} &\ge G-C_0(1+  |u_{ij}|_g^2),
\end{split}\end{equation}
where $G$ and $C_0$ depend on $\gamma$, and $G$ is bounded.  To deal with the bad term $|u_{ij}|^2_g$, we claim that we have the estimate
\begin{equation}\label{goa}
\Theta^{i\ov{j}} \partial_i \partial_{\ov{j}} | \nabla u|^2_g \ge C_1^{-1}  | u_{ij}|^2_g - C,
\end{equation}
for a uniform $C_1>0$. To see this, we compute
\begin{equation}\label{goa2}
\begin{split}
\Theta^{i\ov{j}} \partial_i \partial_{\ov{j}} | \nabla u|^2_g&=  \Theta^{i\ov{j}}g^{p\ov{q}} \left( u_{p \ov{j}} u_{i\ov{q}}  +  u_{pi}u_{\ov{q}\ov{j}} \right) +  \Theta^{i\ov{j}} g^{p\ov{q}} (u_{pi\ov{j}} u_{\ov{q}} + u_{\ov{q} i \ov{j}} u_p)\\
&\geq C'^{-1}|u_{ij}|^2_g+\Theta^{i\ov{j}} g^{p\ov{q}} (u_{pi\ov{j}} u_{\ov{q}} + u_{\ov{q} i \ov{j}} u_p).
\end{split}
\end{equation}
The analog of \eqref{Fell} for equation \eqref{manew} is
$$ \Theta^{i\ov{j}} u_{i\ov{j} p} = F_{p}-  \tilde{g}^{i\ov{j}} (h_{i\ov{j}p}+Z_{i\ov{j}p}).$$
Switching covariant derivatives and using the bound $|\nabla Z|_g\leq C(1+|u_{ij}|_g)$, we obtain
\begin{equation}\label{goa3}
\left|\Theta^{i\ov{j}} g^{p\ov{q}} (u_{pi\ov{j}} u_{\ov{q}} + u_{\ov{q} i \ov{j}} u_p)\right|\leq (2C')^{-1}  |u_{ij}|^2_g+C.
\end{equation}
Combining \eqref{goa2} and \eqref{goa3} proves \eqref{goa}.
Then for $A$ sufficiently large (depending on $\gamma$) we have
$$\Theta^{i\ov{j}} \partial_i \partial_{\ov{j}} (u_{\gamma \ov{\gamma}} + A | \nabla u|^2_g) \ge G.$$

For the next step, we proceed as in \cite{TW4} and consider the metric $\hat{g}_{i\ov{j}}$ on $B_{2R}$ given by
  $\hat{g}_{i\ov{j}} = g_{i\ov{j}}(0).$
Define
 $$\hat{\Theta}^{i\ov{j}} = \frac{1}{n-1} \left( (\tr{\tilde{g}}{\hat{g}}) \hat{g}^{i\ov{j}} - \tilde{g}^{i\ov{j}} \right).$$
By concavity of $\log \det $ and arguing as in \cite{TW4},
we get
\begin{equation} \label{ek4}
\sum_{i,j} \hat{\Theta}^{i\ov{j}} (y) \left( u_{i\ov{j}}(y) - u_{i\ov{j}}(x) \right) \le C'R.
\end{equation}
As in \cite{TW1} for example, we find a set of unit vectors $\gamma_1, \ldots, \gamma_N$ of $\mathbb{C}^n$, containing an orthonormal basis, with the property that
$$\hat{\Theta}^{i\ov{j}}(y) = \sum_{\nu=1}^N \beta_{\nu}(y) (\gamma_{\nu})^i \ov{(\gamma_{\nu})^j},$$
for $\beta_{\nu}$ with $0 < C^{-1}\leq \beta_{\nu} \le C$.  Now define
$$w_{\nu} = u_{\gamma_{\nu} \ov{\gamma_{\nu}}} + A | \nabla u|^2_g.$$
Since we have a uniform bound on $u$ and $\Delta u$ it follows that $| \nabla u|^2_g$ is bounded in $C^{\alpha}$ for any $\alpha$ with $0<\alpha<1$, which we fix once and for all.  Then from (\ref{ek4}), we have
\begin{equation} \label{ek5}
\sum_{\nu=1}^N \beta_{\nu} (w_{\nu} (y) - w_{\nu}(x)) \le C R^{\alpha}.
\end{equation}
On the other hand, from (\ref{ek3}), we have
\begin{equation} \label{ek6}
\Theta^{i\ov{j}} \partial_i \partial_{\ov{j}} w_{\nu} \ge -C.
\end{equation}
From (\ref{ek5}) and (\ref{ek6}) and the arguments of \cite{TW1}, we obtain
$$\Omega(R) \le \delta \Omega(2R) + C R^{\alpha}, \quad \textrm{for } \Omega = \sum_{\nu=1}^N \textrm{osc}_{B_R} w_{\nu},$$
for a uniform $0<\delta<1$ and for $R$ sufficiently small.
Applying  \cite[Lemma 8.23]{GT} we obtain
$$\Omega(R) \le C R^{\kappa},$$
for a small but uniform $\kappa>0$ (with $\kappa<\alpha$) and hence we obtain a H\"older estimate for
the second derivatives $u_{i\ov{j}}$ of $u$ (since we already have a $C^{\alpha}$ estimate for $| \nabla u|^2_g$).

Combining all these estimates, we have proved (assuming \eqref{hmwtype} holds) that there are constants
$C_k$, $k=0,1,\dots$, which depend only on $k, \omega,\omega_0$ and $F$, such that if $u$ solves \eqref{manew}, then
$$\| u \|_{C^k(M,g)} \le C_k,$$
and $$\tilde{g} \ge \frac{1}{C_0} g.$$

To complete the proof of Theorem \ref{thmreduce}, we need to set up a continuity method and establish ``openness'', and prove the uniqueness of the solution. 
These follow from  modifications of the arguments of \cite{FWW2} or \cite{TW4}, which we now briefly explain (uniqueness is also proved in \cite{Po3}, which furthermore contains a calculation of the linearized equation).

We consider the  family of equations $u_t$, $b_t$, for $t \in [0,1]$,
\begin{equation} \label{family}
\left(\omega_h+ \frac{1}{n-1} \left( (\Delta u_t) \omega - \ddbar u_t \right)+*E(u_t)\right)^n = e^{tF+b_t} \omega_h^n,
\end{equation}
with
\begin{equation} \label{MAC2t}
\omega_h+ \frac{1}{n-1} \left( (\Delta u_t) \omega - \ddbar u_t\right)+*E(u_t) >0, \quad \sup_M u_t=0.
\end{equation}
Suppose that we have a solution of (\ref{family}), (\ref{MAC2t}) for $t=\hat{t}$ and write
$$\hat{\omega} = \omega_h + \frac{1}{n-1}( (\Delta u_{\hat{t}}) \omega - \ddbar u_{\hat{t}})+*E(u_{\hat{t}}).$$
Define
\begin{equation*}
\hat{\Theta}^{i\ov{j}} = \frac{1}{n-1} ((\tr{\hat{g}}{g}) g^{i\ov{j}} - \hat{g}^{i\ov{j}}).
\end{equation*}
Consider the linear differential operator
$$L(v)=\hat{\Theta}^{i\ov{j}} v_{i\ov{j}} + \hat{g}^{i\ov{j}} Z(v)_{i\ov{j}}.$$
It is elliptic and its kernel are the constants. Denote by $L^*$ the adjoint of $L$ with respect to the $L^2$ inner product with volume form $\hat{\omega}^n$. We then argue as in \cite{Ga0}.  The index of $L$ is zero and hence the kernel of $L^*$ is one-dimensional, spanned by a smooth function $f$. The maximum principle implies that every
nonzero function in the image of $L$ must change sign, and since $f$ is orthogonal to the image of $L$, it must have constant sign. We can assume that $f\geq 0$. The strong maximum principle then implies that $f>0$, and so we can write $f=e^\sigma$ for a smooth function $\sigma$.  We may and do assume that $\int_M e^{\sigma} \hat{\omega}^n=1$.

Consider the operator
\[
\begin{split}
\Upsilon(v) ={} &  \log \frac{ (\hat{\omega} + \frac{1}{n-1} ((\Delta v)\omega - \ddbar v)+*E(v))^n}{\hat{\omega}^n}   \\ & - \log \left( \int_M e^{\sigma} (\hat{\omega} + \frac{1}{n-1}((\Delta v)\omega - \ddbar v)+*E(v))^n \right),
\end{split}
\]
which maps $C^{2+\alpha}$ functions with integral zero to the space of $C^{\alpha}$ functions $w$ satisfying $\int_M e^w e^{\sigma} \hat{\omega}^n=1$.   Note that the tangent space at $0$ of the latter space consists of $C^{\alpha}$ functions orthogonal to the kernel of $L^*$.
 Since for any $C^{2+\alpha}$ function $\zeta$ we have
$$ \int_M e^{\sigma} (\hat{\Theta}^{i\ov{j}} \zeta_{i\ov{j}}+\hat{g}^{i\ov{j}} Z(\zeta)_{i\ov{j}}) \hat{\omega}^n =\int_M
e^\sigma L(\zeta)\hat{\omega}^n= \int_M \zeta L^*(e^\sigma)\hat{\omega}^n=0,$$
it follows that the  linearization of $\Upsilon$ at $0$ is the operator $L$.  From the Fredholm alternative, $L$ gives an isomorphism of the tangent spaces.  By the Inverse Function Theorem,  we obtain a solution of (\ref{family}) for $t$ close to $\hat{t}$, as required.

Lastly, the uniqueness of the solution  $(u,b)$ of \eqref{manew} follows almost exactly as in \cite{FWW2} or \cite{TW4}.  The only difference is the additional term $*E$ in the equation.  But one only needs to observe that the map $w \mapsto *E(w)$ is linear in $w$ and vanishes at a maximum or minimum of $w$.  The rest of the argument is the same and we refer the reader to Section 6 of \cite{TW4}.

This completes the proof of Theorem \ref{thmreduce}.


\begin{thebibliography}{99}
\bibitem{AS} Alesker, S. and Shelukhin, E., {\em On a uniform estimate for the quaternionic Calabi problem}, Israel J. Math. {\bf 197} (2013), no. 1, 309--327.
\bibitem{Ch} Cherrier, P., {\em \'Equations de Monge-Amp\`ere sur les vari\'et\'es Hermitiennes compactes}, Bull. Sc. Math (2) {\bf 111} (1987), 343--385.
\bibitem{Chi} Chiose, I., {\em The K\"ahler rank of compact complex manifolds}, preprint, arXiv:1308.2043
\bibitem{DK} Dinew, S. and Ko{\l}odziej, S., {\em Liouville and Calabi-Yau type theorems for complex Hessian equations}, preprint, arXiv:1203.3995.
\bibitem{FT}  Fino, A. and Tomassini, A., {\em On astheno-K\"ahler metrics}, J. Lond. Math. Soc. (2) {\bf 83} (2011), no. 2, 290--308.
\bibitem{FWW1} Fu, J., Wang, Z. and Wu, D., {\em Form-type Calabi-Yau equations}, Math. Res. Lett. {\bf 17} (2010), no. 5, 887--903.
\bibitem{FWW2} Fu, J., Wang, Z. and Wu, D., {\em Form-type Calabi-Yau equations on K\"ahler manifolds of nonnegative orthogonal bisectional curvature}, preprint, arXiv: 1010.2022.
\bibitem{FX} Fu, J. and Xiao, J., {\em Relations between the K\"ahler cone and the balanced cone of a K\"ahler manifold}, preprint, arXiv:1203.2978.
\bibitem{FY} Fu, J. and Yau, S.-T., {\em The theory of superstring with flux on non-K\"ahler manifolds and the complex Monge-Amp\`ere equation}, J. Differential Geom. {\bf 78} (2008), no. 3, 369--428.
\bibitem{Ga0} Gauduchon, P., {\em Le th\'eor\`eme de l'excentricit\'e nulle}, C. R. Acad. Sci. Paris S\'er. A-B {\bf 285} (1977), no. 5, A387--390.
\bibitem{Ga} Gauduchon, P., {\em La $1$-forme de torsion d'une vari\'et\'e hermitienne compacte}, Math. Ann. {\bf 267} (1984), 495--518.
\bibitem{GT} Gilbarg, D. and Trudinger, N. S., {\em Elliptic partial differential equations of second order}, Grundlehren der Mathematischen Wissenschaften, Vol. 224. Springer-Verlag, Berlin-New York, 1977.
\bibitem{GL} Guan, B. and Li, Q., {\em Complex Monge-Amp\`ere equations and totally real submanifolds},
Adv. Math. {\bf 225} (2010), no. 3, 1185--1223.
\bibitem{HL} Harvey, F. R. and Lawson, H. B., {\em Geometric plurisubharmonicity and convexity: an introduction}, Adv. Math. {\bf 230} (2012), no. 4-6, 2428--2456.
\bibitem{Hou} Hou, Z., {\em Complex Hessian equation on K\"ahler manifold}, Int. Math. Res. Not. IMRN {\bf 2009}, no. 16, 3098--3111.
\bibitem{HMW} Hou, Z., Ma, X.-N. and Wu, D., {\em A second order estimate for complex Hessian equations on a compact K\"ahler manifold}, Math. Res. Lett. {\bf 17} (2010), no. 3, 547--561.
\bibitem{JY}  Jost, J. and Yau, S.-T., {\em A nonlinear elliptic system for maps from Hermitian to Riemannian manifolds and rigidity theorems in Hermitian geometry}, Acta Math. {\bf 170} (1993), no. 2, 221--254; Correction, Acta Math. {\bf 173} (1994), no. 2, 307.
\bibitem{LY} Li, J. and Yau, S.-T., {\em The existence of supersymmetric string theory with torsion}, J. Differential Geom. {\bf 70} (2005), no. 1, 143--181.
\bibitem{LYZ} Li, J., Yau, S.-T. and Zheng, F., {\em On projectively flat Hermitian manifolds}, Commun. Anal. Geom. {\bf 2} (1994), 103--109.
\bibitem{Ma}  Matsuo, K., {\em Astheno-K\"ahler structures on Calabi-Eckmann manifolds}, Colloq. Math. {\bf 115} (2009), no. 1, 33--39.
\bibitem{MT}  Matsuo, K. and Takahashi, T., {\em On compact astheno-K\"ahler manifolds}, Colloq. Math. {\bf 89} (2001), no. 2, 213--221.
\bibitem{Mi}  Michelsohn, M.L., {\em On the existence of special metrics in complex geometry}, Acta Math. {\bf 149} (1982), no. 3-4, 261--295.
\bibitem{Po} Popovici, D., {\em Deformation limits of projective manifolds: Hodge numbers and strongly Gauduchon metrics}, Invent. Math. (2013).
\bibitem{Po2} Popovici, D., {\em Holomorphic Deformations of Balanced Calabi-Yau $\partial\bar\partial$-Manifolds}, preprint, arXiv:1304.0331.
\bibitem{Po3} Popovici, D., {\em Aeppli Cohomology Classes Associated with Gauduchon Metrics on Compact Complex Manifolds}, preprint, arXiv:1310.3685.
\bibitem{Si} Siu, Y.-T., {\em Lectures on Hermitian-Einstein metrics for stable bundles and K\"ahler-Einstein metrics}, DMV Seminar, 8. Birkh\"auser Verlag, Basel, 1987.
\bibitem{St} Strominger, A., {\em Superstrings with torsion}, Nuclear Phys. B {\bf 274} (1986), no. 2, 253--284.
\bibitem{To} Toma, M., {\em A note on the cone of mobile curves},  C. R. Math. Acad. Sci. Paris {\bf 348} (2010), no. 1-2, 71--73.
\bibitem{TW1} Tosatti, V. and Weinkove, B., {\em Estimates for the complex Monge-Amp\`ere equation on Hermitian and balanced manifolds}, Asian J. Math. {\bf 14} (2010), no.1, 19--40.
\bibitem{TW2} Tosatti, V. and Weinkove, B., {\em The complex Monge-Amp\`ere equation on compact Hermitian manifolds},  J. Amer. Math. Soc. {\bf 23} (2010), no.4, 1187--1195.
\bibitem{TW2a} Tosatti, V. and Weinkove, B., {\em Plurisubharmonic functions and nef classes on complex manifolds}, Proc. Amer. Math. Soc. 140 (2012), 4003--4010
\bibitem{TW3} Tosatti, V. and Weinkove, B., \emph{On the evolution of a Hermitian metric by its Chern-Ricci form}, preprint, arXiv:1201.0312.
\bibitem{TW4} Tosatti, V. and Weinkove, B., {\em The Monge-Amp\`ere equation for $(n-1)$-plurisubharmonic functions on a compact K\"ahler manifold}, preprint, arXiv:1305.7511.
\bibitem{Tr} Trudinger, N.S., {\em Fully nonlinear, uniformly elliptic equations under natural structure conditions}, Trans. Amer. Math. Soc. {\bf 278} (1983), no. 2, 751--769.
\bibitem{Xi} Xiao, J., {\em Weak transcendental holomorphic Morse inequalities on compact K\"ahler manifolds}, preprint, arXiv:1308.2878.
\bibitem{Ya} Yau, S.-T., {\em On the Ricci curvature of a compact K\"ahler manifold and the complex Monge-Amp\`ere equation, I}, Comm. Pure Appl. Math. {\bf 31} (1978), no.3, 339--411.
\end{thebibliography}
\end{document}